\documentclass[12pt,amsfonts]{article}

\usepackage{amsmath,epsfig, amsthm,amsfonts,amsbsy,amssymb,latexsym,amsxtra}
\usepackage[usenames]{color}
\usepackage{pst-infixplot}
\usepackage{pst-plot}
\usepackage{pstricks}
\usepackage{dsfont}

\usepackage[letterpaper=true,colorlinks=true,urlcolor=blue,linkcolor=blue,citecolor=blue,pdfstartview=FitH]{hyperref}
 \usepackage[sort,comma]{natbib}
\begin{document}
\allowdisplaybreaks[1]
\makeatletter \@addtoreset{equation}{section}
\renewcommand{\thesection}{\arabic{section}}
\renewcommand{\theequation}{\thesection.\arabic{equation}}

\newcommand{\bdd}{\hspace*{-0.08in}{\bf.}\hspace*{0.05in}}
\def\para#1{\vskip 0.4\baselineskip\noindent{\bf #1}}
\def\qed{$\qquad \Box$}

\def\ss{{\mathbb S}}
\def\zz{{\mathbb Z}}

\def\chao#1 {\fbox {\footnote {\ }}\ \footnotetext {From Chao:  #1}}
\def\fubao#1 {\fbox {\footnote {\ }}\ \footnotetext {From Fubao: #1}}

\newcommand{\bed}{\begin{displaymath}}
\newcommand{\eed}{\end{displaymath}}
\newcommand{\bea}{\bed\begin{array}{rl}}
\newcommand{\eea}{\end{array}\eed}
 \newcommand{\beq}[1]{\begin{equation} \label{#1}}
\newcommand{\eeq}{\end{equation}}
\newcommand{\disp}{\displaystyle}
\newcommand{\ad}{&\!\!\!\disp}
\newcommand{\aad}{&\disp}
\newcommand{\barray}{\begin{array}{ll}}
\newcommand{\earray}{\end{array}}

\newcommand{\La}{\Lambda}
\newcommand{\la}{\lambda}
\newcommand{\e}{\varepsilon}
\newcommand{\sg}{\sigma}
\newcommand{\lf}{\lfloor}
\newcommand{\rf}{\rfloor}
\newcommand{\wdt}{\widetilde}
\newcommand{\wdh}{\widehat}
\newcommand{\bQ}{{\mathbb Q}}
\newcommand{\al}{\alpha}
\newcommand{\wt}{\widetilde}
\newcommand{\wh }{\widehat}
\newcommand{\R}{\mathbb R}
\newcommand{\E}{{\mathbb E}}
\newcommand{\m}{\mathfrak m}
\renewcommand{\P}{\mathbb P}
\newcommand{\LL}{{\mathcal L}}
\newcommand{\A}{{\mathcal A}}
\newcommand{\B}{{\mathcal B}}
\newcommand{\F}{{\mathcal F}}
\def\G{\mathcal G}
\newcommand{\one}{\mathbf{1}}
\newcommand{\zero}{\mathbf{0}}
\newcommand{\N}{{\cal N}}
\newcommand{\set}[1]{\left\{#1\right\}}
\def\d{\mathrm d}
\newcommand{\tr}{\mathrm {tr}}
\newcommand{\lan}{\langle}
\newcommand{\ran}{\rangle}
\newcommand{\abs}[1]{\left|#1 \right|}
\newcommand{\norm}[1]{\left\|#1 \right\|}
\newtheorem{Theorem}{Theorem}[section]
\newtheorem{Corollary}[Theorem]{Corollary}

\newtheorem{Lemma}[Theorem]{Lemma}
\newtheorem{Note}[Theorem]{Note}
\newtheorem{Proposition}[Theorem]{Proposition}

\theoremstyle{definition}
\newtheorem{Definition}[Theorem]{Definition}
\newtheorem{Remark}[Theorem]{Remark}
\newtheorem{Example}[Theorem]{Example}
\newtheorem{Counterexample}[Theorem]{Counterexample}
\newtheorem{Assumption}[Theorem]{Assumption}

\title{\Large On Feller and  Strong Feller Properties
and Exponential Ergodicity of Regime-Switching Jump Diffusion Processes
with Countable  Regimes\footnote{This research was supported in
 part by the National Natural Science Foundation of China
 under Grant No.11671034.}}
\author{{Fubao Xi}\thanks{School of Mathematics and Statistics,
Beijing Institute of Technology, Beijing 100081, China, xifb@bit.edu.cn.} \and {Chao Zhu}\thanks{Department of Mathematical Sciences, University of Wisconsin-Milwaukee, Milwaukee, WI
53201, zhu@uwm.edu.}}

\maketitle

\begin{abstract}
This work focuses on a class of regime-switching jump diffusion processes, in which the switching component has   countably infinite  many  states or regimes.  The existence and uniqueness of the underlying process are obtained by 
an interlacing procedure. Then   the Feller and strong Feller properties of such processes are derived by  the coupling method and an appropriate Radon-Nikodym derivative.  Finally the paper studies exponential ergodicity of regime-switching jump-diffusion processes.

\bigskip


\bigskip

\noindent{\bf Running Title.} Regime-switching jump diffusion processes

\bigskip

\noindent{\bf 2000 MR Subject Classification.} 60J25, 60J27, 60J60,
60J75.
\end{abstract}


 \baselineskip 18pt
\section{Introduction}\label{sect:introduction}

Jump processes have   become a key model in stochastic analysis  over the
recent years. On one hand this is due to an increasing need for
modeling stochastic processes with jumps in areas ranging from
physics and biology to finance and economics. On the other hand,
there is a more and more profound understanding of  theories and
properties of jump processes. While a general framework is certainly
provided by semimartingale theory, L\'{e}vy processes remain the
basic building blocks. We refer the reader to \cite{APPLEBAUM} for extensive treatments of L\'evy processes.
Meanwhile,
thanks to their ability in incorporating   structural changes,  regime-switching      processes    have attracted many interests lately. See, for example,
 \cite{Xi-08-Feller,MaoY,YZ-10, XiYin-11,ShaoX-14,Wang-14,XiZ-06,CloezH-15,Zhu-10} and the  references therein for investigations  of such processes  and their applications in areas such as inventory control,   ecosystem modeling, manufacturing and production planning, financial engineering,    risk theory,   etc.

Motivated by the increasing need of modeling complex systems,
in which both structural changes and small fluctuations as well as
big spikes coexist and are intertwined,
this paper aims to study regime-switching jump diffusion processes.
Unlike some of the earlier work on  regime-switching jump diffusion processes
such as \cite{Xi-09,ZhuYB-15}, in which the switching component takes value
in a finite state space, this paper allows the switching component
to have an infinite countable state space. This is motivated by the formulation
in the recent work \cite{Shao-15}, in which the switching component
has an infinite countable state space. In the formulation of
\cite{Shao-15},  starting from an arbitrary state,
the switching component  can only switch to a finite neighboring  states (Assumption (A1)). Assumption (A1), together with other conditions, allows the author to derive the existence of a weak solution directly by invoking a result in \cite{Situ-05}.
This paper does not require such a condition.
Instead, certain Lyapunov type condition (condition \eqref{eq:switching-2nd-moment-condition}) is used.
Note that Assumption (A1) of \cite{Shao-15} certainly implies condition \eqref{eq:switching-2nd-moment-condition} but not necessarily the other way around.
As a result of this relaxation, care is needed to establish the existence of a weak solution to the associated stochastic differential equations corresponding to the regime-switching jump diffusion processes.  In this paper,  we use an interlacing procedure together with  exponential killing to construct  a (possibly local) solution to the stochastic differential equations. Condition  \eqref{eq:switching-2nd-moment-condition} together with the growth condition on the coefficients of the stochastic differential equations guarantee that the solution is actually global with no finite explosion time. Finally we establish the pathwise uniqueness result, which gives us the existence and uniqueness of a strong solution  for the associated stochastic differential equations  by virtue of Yamada and Watanabe's result on weak and strong solutions.

Next we use the coupling method to derive  the Feller property for regime-switching jump diffusion processes. The coupling method has been extensively used to study diffusion and jump diffusion processes; see, for example, \cite{ChenLi-89,LindR-86,PriolaW-06,Wang-10} and the references therein. Some earlier work of using the coupling method in the studies of regime-switching (jump) diffusions can be found in \cite{Xi-08,Xi-08-Feller,XiZ-06}. But in these papers, it is assumed that either the switching component is given by a continuous-time Markov chain, resulting the so-called Markovian regime-switching diffusion processes, or the diffusion matrix is independent of the switching component. In this paper, we construct a coupling operator $\wdt \A $ in \eqref{eq-A-coupling-operator}, which can handle the general state-dependent regime-switching jump diffusions. The key idea is that, for the coupled process $(\wdt X, \wdt \La, \wdt Z, \wdt \Xi)$ generated by $\wdt \A$ starting from $(x,k,z,k)$, one needs  to carefully treat the first time when the switching components $\wdt \La$ and $\wdt \Xi$ are different; see the proof of Theorem \ref{thm-Feller} for details.

For the investigation of strong Feller property, we use the idea developed in \cite{Xi-09}.   More precisely, we first show that under certain conditions, the jump diffusion $X^{(k)}$ of \eqref{(EU1)} has strong Feller property. Then we establish the strong Feller property for the  auxiliary process $(V, \psi)$ constructed in  equations \eqref{(GFP6)}--\eqref{(GFP7)}. Next we use the Radon-Nikodym derivative $M_{T}$ of \eqref{(GFP8)} to derive the strong Feller property for the process $(X,\La)$.  In  Section \ref{sect:exp-ergodicity},  as an application of the  strong Feller property, we also  obtain the exponential ergodicity for the regime-switching jump diffusion process $(X,\La)$.  In particular, when the coefficients of the associated stochastic differential equations are linearizable in a neighborhood of $\infty$, we   present some easily verifiable sufficient conditions for exponential ergodicity.

The rest of the paper is arranged as follows.  Section \ref{sect:formulation} presents the precise formulation for regime-switching jump diffusion processes. The standing assumptions are also collected in  Section \ref{sect:formulation}.  The existence and uniqueness results for the associated stochastic differential equations are presented in Section \ref{sect:ExisUniq}. Section \ref{sect-Feller} studies Feller property of regime-switching jump diffusion processes.  Sections \ref{sect-StFP-jump-diffusion} and \ref{sect-StFP} establish strong Feller property for jump diffusion and regime-switching jump diffusion processes, respectively. Section \ref{sect:exp-ergodicity} is devoted to exponential ergodicity of  regime-switching jump diffusion process. Finally, concluding remarks are made in Section \ref{sect:ConRemark}.

\subsection{Formulation}\label{sect:formulation}

Throughout the rest of this paper 
we let
$(\Omega, {\F}, \{{\F}_{t}\}_{t\ge 0}, \, \P )$ be a
complete probability space with a filtration $\{{\F}_{t}\}_{t\ge
0}$ satisfying the usual conditions (i.e., it is 
right continuous and ${\F}_{0}$ contains all $\P $-null sets).
To formulate our model, let $d$ be a positive  integer, and put $\ss
:=\{ 0, 1, 2, \cdots \}$, the totality of nonnegative
   integers.
Let $(X,\La)$ be a right continuous, strong Markov process with
left-hand limits on $\R^d \times \ss$. The first component $X$
satisfies the following stochastic differential-integral equation
\begin{equation}\label{eq:X}
\begin{aligned} \d X(t)&  =
\displaystyle \sigma (X(t),\La(t))\d B(t)+b(X(t),\La(t))\d t\\
& \quad +\displaystyle \int_{U_{0}} c(X(t-),\La(t-),u)
\wdt{N}(\d t,\d u)   
+\displaystyle \int_{U\setminus U_{0}}
c(X(t-),\La(t-),u)N(\d t,\d u),
\end{aligned}
\end{equation}
where $\sg (x,k)$ is $\R^{d \times d}$-valued and $b(x,k)$ and
$c(x,k,u)$ are $\R^{d}$-valued for $x \in \R^{d}$, $k \in \ss$ and
$u \in U$, $\bigl(U, {\B}(U) \bigr)$ is a measurable space,
$B(t)$ is an $\R^{d}$-valued Brownian motion, $N(\d t,\d u)$
(corresponding to a random point process $p(t)$) is a Poisson random
measure independent of $B(t)$, $\wdt{N}( \d t,\d u)=N(\d t,\d u)-\Pi (\d u)\d t$
is the compensated Poisson random measure on $[0,\infty)\times U$,
$\Pi (\cdot)$ is a deterministic $\sg$-finite characteristic
measure on the measurable space $\bigl(U, {\B}(U) \bigr)$, and
$U_{0}$ is a set in ${\B}(U)$ such that $\Pi (U\setminus
U_{0})<\infty$. The second component $\Lambda$ is a discrete random
process with an infinite state space $\ss$ such that
\begin{equation}\label{eq:La}
\P \{\La(t+\Delta)=l | \La(t)=k, X(t)=x \} =\begin{cases}
 q_{kl}(x) \Delta +o(\Delta),
&  \, \, \hbox{if}\, \, k \ne l, \\
1+q_{kk}(x) \Delta +o(\Delta), &  \, \, \hbox{if}\, \, k = l,
\end{cases}
\end{equation}
uniformly in $\R^d$, provided $\Delta \downarrow 0$. As
usual, we assume that  for all $x \in \R^{d}$,  $q_{kl}(x) \ge 0$ for  $l \neq k$ and   $\sum_{l\in \ss}q_{kl}(x)=0$ for all
  $k\in\ss$. For $x \in \R^d$ and $\sg
=(\sg_{ij}) \in \R^{d \times d}$, define
$$|x|=\biggl (\sum_{i=1}^{d} |x_{i}|^{2} \biggr )^{1/2}, \qquad
|\sg|= \biggl ( \sum_{i,j=1}^{d} |\sg_{ij}|^{2} \biggr
)^{1/2}.$$ Define a metric $\lambda (\cdot,\cdot)$ on $\R^{d}
\times \ss$ as $\lambda \bigl((x,m), (y,n) \bigr)=|x-y|+d(m,n)$,
where $d(\cdot,\cdot)$ is the discrete metric on $\ss$ so $d(m,n) =\one_{\{ m\neq n\}} $. Let ${\B}(\R^{d} \times \ss)$ be the Borel $\sigma$-algebra on $\R^{d}
\times \ss$. Then $(\R^{d} \times \ss, \lambda (\cdot,\cdot), {\B}(\R^{d} \times \ss))$ is a locally compact and separable metric
space. For the existence and uniqueness of the strong Markov process
$(X,\La)$ satisfying the system (\ref{eq:X}) and
(\ref{eq:La}), we make the following assumptions.

\begin{Assumption}   \label{I1}
Assume
that $c(x,k,u)$ is ${\B}(\R^d \times \ss)\times {\B}(U)$
measurable, and that for some constant $H>0$,
\begin{align}
\label{linear-growth}
&|b(x,k)|^{2} + |\sigma(x,k)|^{2}+  \int_{U} |c(x,k,u)|^{2} \Pi (\d u) \le H(1+ |x|^{2}),\\
\label{Lip-condition}
& |b(x,k)-b(y,k)|^{2}+|\sg (x,k)-\sg (y,k)|^{2} +  \int_{U} |c(x,k,u)-c(y,k,u)|^{2} \Pi (\d u) \le H |x-y|^{2},
\end{align}
for all $x,y \in \R^{d}$ and   $k \in \ss$.
 \end{Assumption}

 \begin{Assumption}   \label{qH}{\rm
Assume that for all $(x,k)\in \R^{d}\times\ss$,  
 we have
\begin{align}\label{q}
&q_{k}(x) :=-q_{kk}(x) = \sum_{l\in\ss\setminus \{k\}}q_{kl}(x)\le H (k+1),  \\
\label{eq:switching-2nd-moment-condition}
& \sum_{l\in\ss\setminus \{k\}}(f(l)-f(k)) q_{kl} (x) \le H( 1+ |x|^{2} + f(k) ),
\end{align} where the constant $H>0$ is the  same as in
Assumption~\ref{I1} without loss of generality, and the function $f:\ss\mapsto \R_{+}$ is nondecreasing and satisfies $f(m)\to \infty$ as $m \to \infty$.  In addition, assume
 there exists some   $\delta \in (0, 1]$ such that \begin{equation}
\label{eq-q(x)-holder}
\sum_{l\in \ss\setminus\{k\}}| q_{kl}(x) - q_{kl}(y)| \le H |x-y|^{\delta}
\end{equation} for all $k\in \ss$ and $x,y\in \R^{d}$.
}\end{Assumption}

\section{Existence and Uniqueness}\label{sect:ExisUniq}

In this section, we prove that there exists a unique strong solution
to the system  \eqref{eq:X}--\eqref{eq:La}. 

\begin{Theorem}   \label{prop-EU1} Suppose that Assumptions~\ref{I1} and \ref{qH} hold.
Then for each $(x,k) \in \R^{d} \times \ss$, system \eqref{eq:X} and 
\eqref{eq:La} has a unique strong solution
$(X(t),\Lambda(t))$ with $(X(0),\Lambda(0))=(x,k)$.
\end{Theorem}

The proof of this theorem is divided into three steps. In the first step, we construct a solution $(X,\La)$ to  \eqref{eq:X} and  \eqref{eq:La}  with $(X(0),\Lambda(0))=(x,k)$ on the interval $[0, \tau_{\infty})$, where $\tau_{\infty} \le \infty$ is a stopping time to be defined in \eqref{eq-tau-infty-defn}.  After some preparatory work, we then show in the second step that $\tau_{\infty} = \infty$ a.s. Finally we establish pathwise uniqueness for    \eqref{eq:X} and  \eqref{eq:La} in Step 3.
\begin{proof}[Proof of Theorem \ref{prop-EU1} (Step 1)]
Here we use the ``interlacing procedure'' as termed in \cite{APPLEBAUM} to demonstrate that 
 under Assumptions  \ref{I1} and \ref{qH}, the system \eqref{eq:X} and \eqref{eq:La} has a   (possibly local) weak  solution $(X,\La)$. To this end, let the complete filtered probability space $(\Omega, {\F}, \{{\F}_{t}\}_{t\ge 0}, \, \P )$, the $d$-dimensional standard Brownian motion $B$, and the Poisson random measure $N(\cdot, \cdot)$ on $[0, \infty) \times U$ be specified as in Section \ref{sect:formulation}. In addition,  let $\{\xi_{n}\}$ be a sequence of independent mean $1$ exponential random variables  on $(\Omega, {\F}, \{{\F}_{t}\}_{t\ge 0}, \, \P )$ that is independent of  $B$ and   $N$.  Fix some $(x,k)\in \R^{d}\times \ss$ and consider    
 the stochastic differential equation
\begin{equation}
\label{eq-sde-k}
\begin{aligned}
  X^{(k)}(t)  & =x+
\displaystyle \int_{0}^{t}\sigma (X^{(k)}(s),k)\d B(s)+ \int_{0}^{t} b(X^{(k)}(s),k)\d s
  \\ & \quad+\displaystyle  \int_{0}^{t}\int_{U_{0}} c(X^{(k)}(s-),k,u)
\wdt{N}(\d s,\d u)
+\displaystyle \int_{0}^{t} \int_{U\setminus U_{0}}
c(X^{(k)}(s-),k,u)N(\d s,\d u). \end{aligned}
\end{equation}
In view of Theorem IV.9.1 of \cite{IkedaW-89}, such a solution exists and is pathwise unique thanks to \eqref{linear-growth} and \eqref{Lip-condition} of Assumption \ref{I1}. Let \begin{equation}
\label{eq-tau1-defn}
 \tau_{1}= \theta_{1}: = \inf\set{t\ge 0: \int_{0}^{t} q_{k}(X^{(k)}(s))\d s > \xi_{1}}.
\end{equation} 
Then we have \begin{equation}\label{eq1-sw-mechanism}
\P\set{\tau_{1} > t| \F_{t}} = \P\set{\xi_{1} \ge  \int_{0}^{t} q_{k}(X^{(k)}(s))\d s\Big| \F_{t}} = \exp\set{-\int_{0}^{t} q_{k}(X^{(k)}(s))\d s}.
\end{equation}
Thanks to \eqref{q} in Assumption \ref{qH},  
we have $\P\{ \tau_{1} > t\} \ge e^{-H(k+1)t}$ and therefore    $\P(\tau_{1} > 0) =1$.
We define a process $(X,\La) \in \R^{d}\times \ss$ on $[0, \tau_{1}]$ as follows:  \begin{displaymath}
 X(t) = X^{(k)}(t) \text{ for all } t \in [0, \tau_{1}], \text{ and }\La(t)  =    k   \text{ for all } t \in [0, \tau_{1}).
\end{displaymath}
   Moreover, we define $\La(\tau_{1})\in \ss$ according to the probability distribution: \begin{equation}\label{eq-switching-to-location}
 \P\set{\La(\tau_{1}) = l| \F_{\tau_{1}-}} = \dfrac{q_{kl}(X(\tau_{1}-))}{q_{k}(X(\tau_{1}-))} (1- \delta_{kl}) \one_{\{q_{k}(X(\tau_{1}-)) > 0  \}} + \delta_{kl} \one_{\{q_{k}(X(\tau_{1}-)) = 0  \}}.
 \end{equation}
 In general, having determined $(X,\La)$ on $[0, \tau_{n}]$, we let
 \begin{equation}
 \label{eq-theta-n+1-defn}
\theta_{n+1}: = \inf\biggl\{t\ge 0: \int_{0}^{t} q_{\La(\tau_{n})}(X^{(\La(\tau_{n}))}(s))\d s > \xi_{n+1}\biggr\}, \end{equation}   where
 \begin{displaymath}
\begin{aligned}
   X^{(\La(\tau_{n}))}(t)  & : =   X(\tau_{n}) +
\displaystyle \int_{0}^{t} \sigma (X^{(\La(\tau_{n}))}(s),\La(\tau_{n}))\d B(s)+ \int_{0}^{t} b(X^{(\La(\tau_{n}))}(s),\La(\tau_{n}))\d s
  \\ & \qquad+\displaystyle \int_{0}^{t} \int_{U_{0}} c(X^{(\La(\tau_{n}))}(s-),\La(\tau_{n}),u)
\wdt{N}(\d s,\d u) \\
& \qquad
+\displaystyle\int_{0}^{t} \int_{U\setminus U_{0}}
c(X^{(\La(\tau_{n}))}(s-),\La(\tau_{n}),u)N(\d s,\d u).\end{aligned}
\end{displaymath} As argued in  \eqref{eq1-sw-mechanism}, we have
\begin{equation}\label{eq-theta-n+1-distribution} \begin{aligned}
\P\set{\theta_{n+1} > t| \F_{\tau_{n}+t}} & = \P\set{\xi_{n+1} \ge \int_{0}^{t} q_{\La(\tau_{n})}(X^{(\La(\tau_{n}))}(s))\d s\Big|\F_{\tau_{n}+t}}  \\ & = \exp\set{-\int_{0}^{t} q_{\La(\tau_{n})}(X^{(\La(\tau_{n}))}(s))\d s}.
\end{aligned}\end{equation} Again, Assumption \ref{qH} implies that $\P\{\theta_{n+1} > 0 \} =1$.
Then we let \begin{equation}
\label{eq-tau-n+1-defn}
 \tau_{n+1} : = \tau_{n} + \theta_{n+1}
\end{equation} and define $(X,\La)$ on $[\tau_{n}, \tau_{n+1}]$ by
\begin{align}\label{eq-XLa-nth-segment}
X(t) =  X^{(\La(\tau_{n}))}(t -\tau_{n})  \text{ for } t\in  [\tau_{n}, \tau_{n+1}], \, \, \La(t) = \La(\tau_{n})  \text{ for } t\in  [\tau_{n}, \tau_{n+1}), \
\end{align}
and
\begin{equation}\label{eq-sw-n-mechanism}\begin{aligned}  \P& \set{\La(\tau_{n+1}) = l| \F_{\tau_{n+1}-}}  \\  & = \dfrac{q_{\La(\tau_n),l}(X(\tau_{n+1}-))}{q_{\La(\tau_n)}(X(\tau_{n+1}-))} (1- \delta_{\La(\tau_n),l})  \one_{\{q_{\La(\tau_n)}(X(\tau_{n+1}-)) > 0  \}} + \delta_{\La(\tau_n),l} \one_{\{q_{\La(\tau_n)}(X(\tau_{n+1}-)) = 0  \}}.
\end{aligned} \end{equation}

This ``interlacing procedure''
  uniquely    determines a strong Markov   process $(X,\La)\in \R^{d}\times \ss$ for all $t \in[0,\tau_{\infty})$, where
 \begin{equation}
\label{eq-tau-infty-defn}
 \tau_{\infty}=\lim_{n\to \infty}\tau_n.
\end{equation} Since the sequence $\tau_{n}$ is strictly increasing, the limit $\tau_{\infty} \le \infty$ exists. Moreover it follows from \eqref{eq-theta-n+1-distribution}--\eqref{eq-sw-n-mechanism} that the process  $(X,\La)$ satisfies \eqref{eq:X} and \eqref{eq:La} on $[0, \tau_{\infty})$.
\end{proof}

\begin{Remark}
Note that in general condition \eqref{q} alone can not guarantee  that $\tau_{\infty} = \infty$ a.s. To see this, let us consider a continuous-time Markov chain $\La$ with state space $\ss = \{0, 1,\dots, \}$ and $Q$-matrix given by $Q= (q_{kl})$ such that $- q_{kk} =q_{k, (k+1)^{2}} =k+1$ and $q_{kl} =0$ for all   $l \in \ss\setminus\{k, (k+1)^{2}\}$.  For this example, \eqref{q} is satisfied.

Assume $\La(0) =0$, then $\La$ will stay in state 0 for an exponential amount of time with mean $1$ and then switch to state $1$, whose holding time   has exponential distribution with mean $\frac{1}{2}$; it next switches to state $4$, whose holding time   is exponentially distributed with mean $\frac{1}{5}$; and then switches to state $26$, whose holding time   is exponential with mean $\frac{1}{27}$; and so on. It is then clear that $\P(\tau_{\infty} < \infty) =1$. Of course, we can easily check that condition  \eqref{eq:switching-2nd-moment-condition} can not be satisfied for this example.
\end{Remark}

\begin{Remark}\label{HkvsH}
However, if the upper bound $H(k+1)$ in \eqref{q} of Assumption \ref{qH} is replaced by $H$, then we have $\tau_{\infty} = \infty$ a.s. and therefore the proof of Theorem \ref{prop-EU1} can be much simplified. Indeed, with the uniform upper bound, we have $\P\{\theta_{k} > t\}  \ge e^{- Ht}$ for all $k \in \mathbb N$ and $t > 0$ and hence \begin{equation}\label{Hfinite}
\begin{aligned}
\P\{ \tau_{\infty} =\infty\}  &  \ge \P \bigl\{  \{\theta_{k} > t \} \text{ i.o.}\bigr\} =  \P \Biggl\{ \bigcap_{m=1}^{\infty} \bigcup_{k=m}^{\infty} \{\theta_{k} > t \} \Biggr\}\\  & = \lim_{m\to\infty}\P \Biggl\{  \bigcup_{k=m}^{\infty} \{\theta_{k} > t \} \Biggr\}
 \ge \limsup_{m\to\infty}\P \{\theta_{m} > t   \}  \ge e^{-Ht}.
\end{aligned}
\end{equation}  
 Letting $t \downarrow 0$ yields that $ \P\{ \tau_{\infty} = \infty\} =1$. Thus the ``interlacing procedure'' directly leads to the existence of a solution $(X,\La)$ to \eqref{eq:X}--\eqref{eq:La} for {\em all} $t \in [0,\infty)$.
\end{Remark}

To proceed, we construct a family of disjoint intervals $\{\Delta_{ij}(x): i,j \in \ss\}$ on the positive half real line as follows:
\begin{align*}
  \Delta_{01}(x)  & = [0, q_{01}(x)),   \\
   \Delta_{02}(x)  & = [q_{01}(x)), q_{01}(x) + q_{02}(x)), \\
    &\ \   \vdots  \\
    \Delta_{10}(x) & = [q_{0}(x), q_{0}(x) + q_{10}(x)), \\
    \Delta_{12}(x) & = [q_{0}(x) + q_{10}(x)), q_{0}(x) + q_{10}(x) + q_{12}(x)), \\
    & \ \   \vdots   \\
      \Delta_{20}(x) & = [q_{0}(x) + q_{1}(x),q_{0}(x) +  q_{1}(x) + q_{20}(x)), \\
      & \ \   \vdots   \\
\end{align*}
 where for convenience of notations, we set 
 $\Delta_{ij}(x) = \emptyset$ if $q_{ij}(x) = 0$, $i \not= j$.  Note that for each $x\in \R^{n}$,  $\{\Delta_{ij}(x): i,j \in \ss\}$ are disjoint intervals, and the length of the interval $\Delta_{ij}(x)$ is equal to $q_{ij}(x)$, which is bounded above by $Hi$ thanks to Assumption~\ref{qH}.
We then define a function $h$: $\R^{d} \times \ss \times \R_{+} \to \R$ by
\begin{equation}\label{eq-h-fn}h(x,k,r)=\sum_{l \in \ss}(l-k){\mathbf{1}}_{\Delta_{kl}(x)}(r).\end{equation} That is,
 for each $x\in \R^{d}$ and $ k \in \ss$,  we set $h(x,k,r)=l-k$ if $r \in\Delta_{kl}(x)$ for some $l\neq k$;
otherwise $h(x,k,r)=0$.

\begin{Proposition}\label{prop-dynkin}  Let Assumptions \ref{I1} and \ref{qH} hold. For any $f \in C_{c}^{2}(\R^{d}\times \ss)$, we have
 \begin{equation}\label{eq-dynkin}
\E_{x,k} [f(X(t\wedge \tau_{\infty}), \La(t\wedge \tau_{\infty}))] = f(x,k) + \E_{x,k} \biggl[\int_{0}^{t\wedge \tau_{\infty}} \A f(X(s),\La(s)) \d s \biggr],
\end{equation} where  \begin{equation}
\label{eq-operator}
\A f(x,k) : = \LL_{k}f(x,k) +  Q(x)f(x,k),
\end{equation}  with
\begin{align}
\label{eq-Lk-operator-defn}& {\LL}_{k} f(x,k)
: =\frac {1}{2}\tr\bigl(a(x,k)\nabla^2 f(x,k)\bigr)+\langle b(x,k),
\nabla f(x,k)\rangle\\
\nonumber &  \qquad \qquad \quad + \int_{U}
\bigl(f(x+c(x,k,u),k)-f(x,k)- \langle \nabla f(x,k), c(x,k,u)\rangle
{\mathbf{1}}_{\{ u \in U_{0}\}}\bigr)\Pi(\d u), \\
\label{eq-Qx-operator}
& Q(x)f(x,k)  : = \sum_{j\in \ss} q_{kj}(x) [f(x,j) - f(x,k)]= \int_{[0,\infty)}[ f(x, k+ h(x,k,z)) - f(x,k)] \m(\d z).
\end{align}
\end{Proposition}

\begin{proof}
    Put $\lambda(t) : = \int_{0}^{t}q_{\La(s)} (X(s))\d s$
 and  $n(t): = \max\{n \in \mathbb N: \xi_{1} + \dots+ \xi_{n} \le \lambda(t) \}$  for all $t \in [0, \tau_{\infty})$,
where $\{ \xi_{n}, n=1,2,\dots\}$ is a sequence of independent exponential random variables with mean 1. 
Then in view of \eqref{eq-tau1-defn}, \eqref{eq1-sw-mechanism}, \eqref{eq-theta-n+1-defn},  \eqref{eq-theta-n+1-distribution}, and \eqref{eq-tau-n+1-defn}, the process $\{n(t\wedge\tau_{\infty}), t\ge 0\}$
is a  counting process that counts the number of switches for the component $\La$.  We can regard $n(\cdot) $ as a nonhomogeneous Poisson process with random intensity function $q_{\La(t)} (X(t))$, $t \in [0, \tau_{\infty})$.

Now for any $s < t \in   [0, \tau_{\infty})$ and  $A\in \B( \ss)$, let
$$\mathfrak p((s,t]\times A) : = \sum_{u \in (s, t]} \one_{ \{ \La(u) \neq \La(u-),   \La(u)  \in A\}} \text{ and }  \mathfrak p(t,A) : =  \mathfrak p((0,t] \times A).$$
Then we have $\mathfrak p(t\wedge\tau_{\infty}, \ss) =   n(t\wedge\tau_{\infty})$   and
\begin{equation}
\label{eq0-La-sde}\begin{aligned}
 \La(t\wedge\tau_{\infty}) & =  \La(0) + \sum_{k=1}^{\infty}[\La(\tau_{k})- \La(\tau_{k}-)] \one_{\{\tau_{k}\le t\wedge\tau_{\infty}\}} \\ & = \La(0) + \int_{0}^{t\wedge\tau_{\infty}} \int_{\ss}[l- \La(s-)] \, \mathfrak p (\d s, \d l).
\end{aligned}\end{equation}

  We can also define a Poisson random measure $N_{1}(\cdot,\cdot)$ on $[0, \infty) \times \R_{+}$ by  $$N_{1}(t\wedge\tau_{\infty}, B) : = \sum_{l\in \ss \cap B} \mathfrak p(t\wedge\tau_{\infty},l), \ \text{ for all } t \ge 0 \text{ and } B \in \mathcal B(\R_{+}).$$

Observe that for any $(x,k)\in \R^{d}\times \ss$ and $l \in \ss\backslash\{k\}$, we have \begin{displaymath}
 \m\{r\in [0,\infty): h(x,k,r) \neq 0 \} = q_{k}(x) \text{ and }\m\{ r \in [0,\infty): h(x,k,r) = l-k\} = q_{kl}(x),
\end{displaymath} where $\m$ is the Lebesgue measure on $ \R_{+}$. Therefore we can rewrite \eqref{eq-sw-n-mechanism} and  \eqref{eq0-La-sde} as
\begin{equation}\label{eq:La-SDE}
 \La(t\wedge\tau_{\infty})=\La(0) + \int_{0}^{t\wedge\tau_{\infty}}\int_{\R_{+}}h(X(s-),\La(s-),r){N}_{1}(\d s,\d r).
\end{equation}
Then we can use the same argument as that in the proof of Lemma 3 on p. 105 of \cite{Skorohod-89} to  show that  for any $f\in C^{2}(\R^{d}\times \ss)$, we have
\begin{align*}
f &(X(t\wedge\tau_{\infty}), \La(t\wedge\tau_{\infty})) \\ & = f(x,k) + \int_{0}^{t\wedge\tau_{\infty}} \A f(X(s),\La(s))\d s +
\int_{0}^{t\wedge\tau_{\infty}} \nabla f(X(s), \La(s)) \cdot \sigma(X(s),\La(s)) \d B(s) \\
 & \  + \int_{0}^{t\wedge\tau_{\infty}} \int_{U_{0}}[f(X(s-) + c(X(s-),\La(s-),u),\La(s-)) - f(X(s-),\La(s-)) ] \wdt N(\d s, \d u)\\
 & \ +  \int_{0}^{t\wedge\tau_{\infty}} \int_{\R_{+}}[f(X(s-), \La(s-) + h(X(s-),\La(s-),r)) - f(X(s-),\La(s-)) ] \wdt N_{1}(\d s, \d r),
\end{align*} where $\wdt N_{1}(\d s, \d r) : =   N_{1}(\d s, \d r) - \d s \m(\d r)$.  In particular, \eqref{eq-dynkin} follows.
\end{proof}

We immediately have the following corollary from Proposition \ref{prop-dynkin}.
\begin{Corollary}\label{cor-generator} Suppose Assumptions \ref{I1} and \ref{qH}.
Then the extended  generator of the process $(X,\La)$ is given by $\A$ of \eqref{eq-operator} on the temporal interval $[0,\tau_{\infty})$.
\end{Corollary}


 \begin{proof}[Proof of Theorem \ref{prop-EU1} (Step 2)] Now we are ready to show that $\tau_{\infty} = \infty$ a.s. and hence the    ``interlacing procedure'' presented in Step 1
 actually      determines a strong Markov   process $(X,\La)\in \R^{d}\times \ss$ for all $t \in[0,\infty)$. To this end, fix $(X(0),\La(0)) = (x,k)\in \R^{d}\times \ss$ as in Step 1 and for any $m \ge k+1$, we denote by $\wdt{\tau}_m :=\inf \{t\ge 0: \La(t)\ge m\}$ the first exit time for the $\La$ component from the finite set $\{0, 1, \dots, m-1\}$.  Let $A^{c}: =\{ \omega\in \Omega: \tau_{\infty} > \wdt \tau_{m} \text{  for all }m \ge k+ 1\}$ and $A : = \{ \omega\in \Omega: \tau_{\infty} \le \wdt \tau_{m_{0}} \text{ for some }m_{0} \ge k+1\}$. Then we have
  \begin{equation}
\label{eq-tau-infty=infty}
\P\{\tau_{\infty} = \infty\} = \P\{ \tau_{\infty} = \infty |A^{c}\} \P(A^{c}) + \P\{ \tau_{\infty} = \infty| A\} \P(A) .
\end{equation} 

Let us first show that $\P\{ \tau_{\infty} = \infty| A\}  =1$. To this end,
we note that on the event $A$, we have $\La(\tau_{n}) \in \{0,1,\dots, m_{0}-1\}$ and hence by \eqref{q}, $$q_{\La(\tau_{n})}(X^{(\La(\tau_{n}))}(s)) \le H (\La(\tau_{n}) + 1) \le Hm_{0}, \text{ for all }n =1, 2,\dots \text{ and } s\ge 0.$$ Then it follows from   \eqref{eq-theta-n+1-distribution} that for all $n =0, 1,\dots$ \begin{align*}
\P\set{\theta_{n+1} > t| \F_{\tau_{n}+t}}  & = \exp\set{-\int_{0}^{t} q_{\La(\tau_{n})}(X^{(\La(\tau_{n}))}(s))\d s} \\ & \ge   \one_{A}\exp\set{-\int_{0}^{t} q_{\La(\tau_{n})}(X^{(\La(\tau_{n}))}(s))\d s} \\ &  \ge e^{-Hm_{0}t} \one_{A}.
\end{align*} Taking expectations on both sides yields $\P(\theta_{n+1} > t) \ge e^{-Hm_{0}t} \P(A)$ and hence $\P \{ \theta_{n+1} > t| A\} \ge e^{-Hm_{0}t} $. Thus, as argued  in \eqref{Hfinite}, we obtain that for any $t > 0$,
\begin{align*}
\P\{ \tau_{\infty} =\infty|A\} & \ge \P \bigl\{\{\theta_{n} > t \} \text{ i.o.}  |A \bigr\}
\ge \limsup_{m\to\infty}\P \{\theta_{m} > t  |A \}  \ge e^{-Hm_{0}t}.
\end{align*} 
Letting $t \downarrow 0$ yields that $ \P\{ \tau_{\infty} = \infty|A\} =1$.

If $\P(A) =1$ or $\P(A^{c}) =0$, then \eqref{eq-tau-infty=infty} implies that $\P\{\tau_{\infty} = \infty \} =1$ and the proof is complete. Therefore it remains to consider the case when $\P(A^{c}) >0$.  Denote $\wdt\tau_{\infty}: =\lim_{m\to \infty}\wdt \tau_{m}$.  Note that $A^{c}= \{ \tau_{\infty} \ge \wdt \tau_{\infty}\}$. Thus $\P\{\tau_{\infty} = \infty |A^{c} \}\ge \P\{\wdt \tau_{\infty} = \infty |A^{c} \}$ and hence \eqref{eq-tau-infty=infty}  will hold true if we can show  that \begin{equation}
\label{eq-tilde-tau-infty=infty}
\P\{\wdt \tau_{\infty} = \infty |A^{c} \} =1.
\end{equation} Assume on the contrary that \eqref{eq-tilde-tau-infty=infty} was false, then there would exist a $T > 0$ such that
\begin{displaymath}
\delta: = \P\{\wdt \tau_{\infty} \le T, A^{c} \} > 0.
\end{displaymath}
Let $f:\ss\mapsto \R_{+}$ be as in Assumption \ref{qH}.  Then by virtue of  the Dynkin  formula  \eqref{eq-dynkin},
we have for any $m \ge k+1$,
\begin{align*}
 f(k) & = \E[e^{-H(T\wedge \tau_{\infty}\wedge \wdt \tau_{m})} f(\La(T\wedge \tau_{\infty}\wedge \wdt \tau_{m}))]  \\
               & \qquad + \E\biggl[ \int_{0}^{T\wedge \tau_{\infty}\wedge \wdt \tau_{m}} e^{-Hs}\biggl(H f(\La(s)) - \sum_{l\in \ss} q_{\La(s), l}(X(s)) [f(l) - f(\La(s))] \biggr) \d s\biggr] \\
 &  \ge \E[e^{-H(T\wedge \tau_{\infty}\wedge \wdt \tau_{m})} f(\La(T\wedge \tau_{\infty}\wedge \wdt \tau_{m}))]  \\
               & \qquad +  \E\biggl[ \int_{0}^{T\wedge \tau_{\infty}\wedge \wdt \tau_{m}} e^{-Hs} [H f(\La(s)) - H (1+ |X(s)|^{2} + f(\La(s))) ] \d s \biggr]  \\
               & \ge  \E[e^{-H(T\wedge \tau_{\infty}\wedge \wdt \tau_{m})} f(\La(T\wedge \tau_{\infty}\wedge \wdt \tau_{m}))],
\end{align*} where the  first inequality above follows from \eqref{eq:switching-2nd-moment-condition} in Assumption \ref{qH}.
Consequently we have
\begin{equation}
\label{eq-exp-HT-f(k)}
\begin{aligned}
 e^{HT} f(k)& \ge   \E[ f(\La(T\wedge \tau_{\infty}\wedge \wdt \tau_{m}))] \ge \E[f(\La(\wdt\tau_{m})) \one_{\{\wdt\tau_{m} \le T\wedge \tau_{\infty}\}}] \\ & \ge f(m) \P\{\wdt\tau_{m} \le T\wedge \tau_{\infty}\}   \ge f(m) \P\{\wdt\tau_{m} \le T\wedge \tau_{\infty}, A^{c}\} \\ & \ge  f(m) \P\{\wdt\tau_{\infty} \le T\wedge \tau_{\infty}, A^{c}\},
\end{aligned}
\end{equation}where the third inequality follows from the facts that $\La(\wdt \tau_{m}) \ge m$ and that $f $ is nondecreasing, and the last inequality follows from the fact that $\wdt \tau_{m} \uparrow \wdt \tau_{\infty} $.  Recall that $A^{c} = \{\tau_{\infty} \ge \wdt \tau_{\infty} \}$. Thus \begin{align*}
 \P\{\wdt\tau_{\infty} \le T\wedge \tau_{\infty}, A^{c}\} &= \P \{\wdt\tau_{\infty} \le T\wedge \tau_{\infty}, \wdt \tau_{\infty} \le   \tau_{\infty}  \} \\ &\ge \P \{\wdt\tau_{\infty} \le T, \wdt \tau_{\infty} \le \tau_{\infty} \} = \P\{\wdt\tau_{\infty} \le T, A^{c} \} = \delta> 0.
\end{align*}
Using this observation in \eqref{eq-exp-HT-f(k)} yields $\infty > e^{HT} f(k) \ge f(m) \delta \to \infty$ as $m\to \infty$, thanks to the fact that $f(m) \to \infty$ as $m\to \infty$. This is a contradiction.
This establishes \eqref{eq-tilde-tau-infty=infty} and therefore completes the proof. \end{proof}

\begin{Lemma}\label{lem-no-explosion} Under Assumptions \ref{I1} and \ref{qH},    the process $(X,\La)$ has no finite explosion time with probability one; that is, $\P\{ T_{\infty} =\infty\} =1$, where $$T_{\infty}: =\lim_{n\to \infty} T_{n}, \text{ and }  T_{n} : = \inf\{ t \ge0: |X(t)| \vee \Lambda(t) \ge n\}.$$
\end{Lemma}
\begin{proof}
Consider the function $V(x,k) : = |x|^{2 } + f(k)$, where the function $f: \ss\mapsto \R_{+}$ is as in Assumption \ref{qH}. Then we have from Assumptions \ref{I1} and \ref{qH} that \begin{align*}
\A V(x,k)  &= 2x\cdot b(x,k) + \frac12 \tr (\sigma\sigma'(x,k) 2I)  + \sum_{l \in \ss} q_{kl}(x) [f(l)-f(k)]\\
 &  \quad + \int_{U} [|x+ c(x,k,u)|^{2} - |x|^{2} -2 x\cdot c(x,k,u)\one_{U_{0}}(u)] \Pi(\d u) \\
 & \le |x|^{2} + |b(x,k) |^{2} + |\sigma(x,k) |^{2}  + H(1+ |x|^{2} + k)  \\
 & \quad + \int_{U} |c(x,k,u)|^{2} \Pi(\d u)+  \int_{U\setminus U_{0}} 2x\cdot c(x,k,u) \Pi(\d u)\\
 & \le K (1+ |x|^{2} + f( k) ) = K (1 + V(x,k)),
\end{align*} where $K$ is a positive constant. Then the conclusion follows from Theorem 2.1 of \cite{MeynT-93III}.
\end{proof}



\begin{proof}[Proof of Theorem \ref{prop-EU1} (Step 3)] Finally we show that pathwise uniqueness for \eqref{eq:X}--\eqref{eq:La} holds. This, together with the existence result established in Steps 1 and 2, then implies that \eqref{eq:X}--\eqref{eq:La} has a unique strong solution $(X,\La)$.

Suppose $(X,\La)$ and $(\wdt X, \wdt \La)$ are two solutions to \eqref{eq:X}--\eqref{eq:La} starting from the same initial condition $(x,k) \in \R^{d}\times \ss$. Then we have
\begin{align*}
& \wdt X(t) - X(t)\\  &\ \ = \int_{0}^{t} [b(\wdt X(s),\wdt\La(s)) - b(X(s),\La(s))]\d s 
 + \int_{0}^{t}  [\sigma(\wdt X(s),\wdt\La(s)) - \sigma(X(s),\La(s))]\d W(s) \\ & \qquad + \int_{U_{0}} [c(\wdt X(s-),\wdt \La(s-), z) - c(X(s-),\La(s-),z)] \wdt N(\d s, \d u) \\
& \qquad +\int_{U\setminus U_{0}} [c(\wdt X(s-),\wdt \La(s-), z) - c(X(s-),\La(s-),z)]   N(\d s, \d u),  \intertext{and}
 &\wdt \La(t) - \La(t) =\int_{0}^{t} \int_{\R_{+}}[h(\wdt X(s-),\wdt \La(s-), z) - h(X(s-),\La(s-),z)] N_{1} (\d s, \d z).
\end{align*}

Let $\zeta:=\inf\{ t \ge 0: \La(t) \neq \wdt \La(t)\}$ be the first time when the discrete components differ from each other and define $T_{R} : =\inf\{t \ge 0: |\wdt X(t)| \vee |X(t)| \vee \wdt \La(t) \vee \La(t) \ge R\}$ for $R > 0$. Lemma \ref{lem-no-explosion} implies that $T_{R}\to \infty$ a.s. as $R\to \infty$.
 Note that $\wdt \La(s) = \La(s)$ for all $s < \zeta$.
Detailed computations using \eqref{Lip-condition} in Assumption \ref{I1}  reveal that
\begin{align*}
 \E& [ |\wdt X(t\wedge \zeta\wedge T_{R}) - X(t\wedge \zeta\wedge T_{R})|^{2} ] \\& = \E\biggl[\int_{0}^{t\wedge \zeta\wedge T_{R}} \biggl( 2(\wdt X(s)- X(s)) \cdot (b(\wdt X(s),\La(s)) - b(X(s),\La(s)))  \\
 &\qquad + |\sigma(\wdt X(s),\La(s)) - \sigma(X(s),\La(s))|^{2} \\
 & \qquad + \int_{U} |c(\wdt X(s-), \La(s-),u) - c( X(s-), \La(s-),u) |^{2} \Pi(\d u) \\
 & \qquad  +\int_{U\setminus U_{0}} 2 (\wdt X(s)- X(s)) \cdot (c(\wdt X(s-), \La(s-),u) - c( X(s-), \La(s-),u)) \Pi(\d u) \biggr)\d s \biggr] \\
 & \le K \E\biggl[\int_{0}^{t\wedge \zeta\wedge T_{R}}  |\wdt X(s)- X(s)|^{2} \d s\biggr] \\
 & = K \int_{0}^{t} \E[ |\wdt X(s\wedge \zeta\wedge T_{R})- X(s\wedge \zeta\wedge T_{R})|^{2} ]\d s,
\end{align*}
where $K$ is a positive constant. Applying Gronwall's inequality, we see that $$ \E [ |\wdt X(t\wedge \zeta\wedge T_{R}) - X(t\wedge \zeta\wedge T_{R})|^{2} ]= 0$$ for all $R > 0$ and thus $ \E [ |\wdt X(t\wedge \zeta) - X(t\wedge \zeta)|^{2} ]= 0$, which, in turn, implies that \begin{equation}
\label{eq-difference-1st-moment} \E [ |\wdt X(t\wedge \zeta) - X(t\wedge \zeta)|  ]= 0 \text{ and }
 \E [ |\wdt X(t\wedge \zeta) - X(t\wedge \zeta)|^{\delta} ]= 0,
\end{equation} where $\delta \in (0, 1]$ is the H\"{o}lder constant in \eqref{eq-q(x)-holder}.

Note that $\zeta \le t $ if and only if $\wdt \La(t\wedge \tau) - \La(t\wedge \tau)\neq 0$. Therefore it follows that
\begin{align*}
\P&\{\zeta \le t\}    =\E[\one_{ \{\wdt \La(t\wedge \tau) - \La(t\wedge \tau)\neq 0\}}]   \\
 & = \E \biggl[\int_{0}^{t\wedge \zeta} \int_{\R_{+}} (\one_{\{\wdt \La(s-) - \La(s-) + h(\wdt X(s-), \La(s-), z) - h(X(s-),\La(s-), z) \neq 0 \}}  -  \one_{\{\wdt \La(s-) - \La(s-)\neq 0 \}})  \m(\d z) \d s\biggr] \\
 & = \E\biggl[\int_{0}^{t\wedge \zeta} \int_{\R_{+}}  \one_{\{h(\wdt X(s-), \La(s-), z) - h(X(s-),\La(s-), z) \neq 0 \}} \m(\d z) \d s\biggr] \\
 & \le  \E\biggl[\int_{0}^{t\wedge \zeta} \sum_{l \in \ss, l \neq \La(s-)} | q_{\La(s-),l}(\wdt X(s-)) - q_{\La(s-),l}(X(s-))| \d s\biggr]\\
 & \le \kappa \E \biggl[\int_{0}^{t\wedge \zeta} | \wdt X(s-) ) -  X(s-) |^{\delta} \d s\biggr]  = \kappa \int_{0}^{t} \E[|\wdt X(s\wedge \zeta) - X(s\wedge \zeta) |^{\delta}] \d s =0,
\end{align*}  where the second inequality follows from \eqref{eq-q(x)-holder}.
In particular, it follows that      $$\E[\one_{\{\wdt \La(t) \neq \La(t) \}}] =0.$$ 
Note also that  $\wdt X(t) - X(t)$ is integrable and hence it follows that $\E[ |\wdt X(t) - X(t) |\one_{\{\zeta \le t \}}]  =0$.
Now we can compute
\begin{align*}
\E[|\wdt X(t) - X(t) | ] & =\E[ |\wdt X(t) - X(t) |\one_{\{\zeta > t \}}] + \E[ |\wdt X(t) - X(t) |\one_{\{\zeta \le t \}}]   \\
 & =\E[ |\wdt X(t\wedge \zeta) - X(t\wedge \zeta) |\one_{\{\zeta > t \}}] + \E[ |\wdt X(t) - X(t) |\one_{\{\zeta \le t \}}]\\
 & \le \E[ |\wdt X(t\wedge \zeta) - X(t\wedge \zeta) |] + 0 \\
 & =0,
\end{align*}
Recall that $\lambda((x,m),(y,n)): = |x-y| + \one_{\{m\neq n\}}$ is a metric on $\R^{d}\times \ss$. Hence we have shown that $$\E[\lambda((\wdt X(t), \wdt\La(t)), ( X(t),\La(t)))] = 0 \text{ for all }t \ge 0.$$ Thus $\P \{ (\wdt X(t), \wdt\La(t))=  ( X(t),\La(t))\} =1$ for all $t\ge 0$. This, together with the fact that the sample paths of $(X,\La)$ are right continuous, implies the desired pathwise uniqueness result.
\end{proof}

We finish the section with some moment estimates for the solution $(X,\La)$ of  \eqref{eq:X}--\eqref{eq:La}.
\begin{Proposition}\label{prop-moment-estimate}
Suppose Assumptions \ref{I1} and \ref{qH}.  Then we have for any $T\ge 0$ \begin{equation}\label{eq:X^2-moment}
\disp   \E_{x,k}\left[ \sup_{0\le t\le T} |X(t)|^{2}\right]  \le C_{1}, \end{equation} where $C_{1} = C_{1} (x,T,H)$ is a positive constant.
Assume in addition   that \begin{equation}\label{eq-2-17}
\biggl(\sum_{l\neq k} (l-k) q_{kl}(x)\biggr)^{2} \le H (1+ |x|^{2} + k^{2})
\end{equation} for all $(x,k) \in \R^{d}\times \ss$. Then for any $T\ge 0$, we have
\begin{equation}\label{eq-2nd-moment}
\E_{x,k}\biggl[ \sup_{0\le t \le T} (|X(t)|^{2}+ \La(t)^{2})\biggr] \le C_{2},
\end{equation} where $C_{2}= C_{2}(x, k, T, H)$ is a positive constant.
\end{Proposition}

\begin{proof}
We notice that  the standard arguments using the linear growth condition \eqref{linear-growth} in Assumption \ref{I1} and  the BDG  inequality (see, for example, the proof of Lemma 3.1 in \cite{ZhuYB-15}) allow us to derive
\begin{equation}\label{eq:X^2-estimate}
\disp   \E_{x,k}\left[ \sup_{0\le t\le T} |X(t)|^{2}\right]  \le K_{1}   + K_2 \int^T_0 \E_{x,k}\left[\sup_{1\le u\le s} | X(u)|^{ 2} \right] \d s, \end{equation}
where $K_{1}, K_{2}$ are positive constants depending only on $x, H$, and $ T$. Then \eqref{eq:X^2-moment} follows from  Gronwall's inequality.

  It remains to establish \eqref{eq-2nd-moment} under the additional condition \eqref{eq-2-17}. Since \begin{align*}
\La(t) &  = k + \int_{0}^{t} \int_{\R_{+}} h(X(s-),\La(s-), r) N_{1}(\d s, \d r) \\
          &  = k + \int_{0}^{t} \int_{\R_{+}} h(X(s-),\La(s-), r) \wdt N_{1}(\d s, \d r)   +   \int_{0}^{t} \int_{\R_{+}} h(X(s-),\La(s-), r) \m(\d r) \d s,
\end{align*}
we can use the BDG and H\"older inequalities to compute 
\begin{align}\label{eq:La^2-estimate}
\nonumber \E_{x,k}\biggl[ \sup_{0\le t \le T}  \La(t)^{2}\biggr] & \le 3 k^{2} + 3 \E_{x,k} \biggl[ \sup_{0\le t \le T} \biggl( \int_{0}^{t} \int_{\R_{+}} h(X(s-),\La(s-), r) \wdt N_{1}(\d s, \d r) \biggr)^{2} \biggr] \\
 \nonumber &  \quad + 3 \E_{x,k} \biggl[ \sup_{0\le t \le T} \biggl( \int_{0}^{t} \int_{\R_{+}} h(X(s-),\La(s-), r)   \m(\d r) \d s \biggr)^{2} \biggr] \\
\nonumber  &\le 3 k^{2} +3 \E_{x,k} \biggl[  \int_{0}^{T} \int_{\R_{+}} h^{2}(X(s-),\La(s-), r) \m(\d r) \d s \biggr]\\
\nonumber  & \quad + 3 \E_{x,k} \Biggl[ \Biggl(  \int_{0}^{T} \sum_{l\in \ss, l\neq \La(s-)} (l-\La(s-)) q_{\La(s-),l}(X(s-)) \d s  \Biggr)^{2} \Biggr] \\
\nonumber  & \le 3 k^{2} + 3 \E_{x,k} \biggl[    \int_{0}^{T} \sum_{l\in \ss, l\neq \La(s-)} (l- \La(s-))^{2} q_{\La(s-),l}(X(s-)) \d s   \biggr] \\
\nonumber  & \quad + 3 \E_{x,k} \biggl[   \int_{0}^{T} 1^{2} \d s \int_{0}^{T} \biggl(\sum_{l\in \ss, l\neq \La(s-)} (l- \La(s-))^{2} q_{\La(s-),l}(X(s-)) \biggr)^{2}\d s  \biggr]  \\
\nonumber  & \le 3 k^{2} + 3 H (1+ T)  \E_{x,k} \biggl[  \int_{0}^{T} (1+ |X(s-)|^{2} + \La(s-)^{2} ) \d s\biggr]\\
 & \le 3 k^{2} + 3 H (1+ T)  \E_{x,k} \biggl[  \int_{0}^{T} \Bigl[ 1+\sup_{0\le u \le s} (|X(u-)|^{2} + \La(u-)^{2}) \Bigr] \d s\biggr],
\end{align}
where we used \eqref{eq:switching-2nd-moment-condition} and \eqref{eq-2-17} to derive the second last inequality.
Then \eqref{eq-2nd-moment} follows from a combination of  \eqref{eq:X^2-estimate} and \eqref{eq:La^2-estimate} and Gronwall's inequality.
\end{proof}
\section{Feller Property}\label{sect-Feller}
We make the following assumption throughout this section:
\begin{Assumption}   \label{qkappa}{\rm
Suppose that  for all $x,z\in \R^{d} $ and $k \in \ss$, we have 
\begin{equation}
\label{eq-c0-nu-lip}
\int_{U} \abs{c(x,k,u) - c(z,k,u)}\Pi(\d u) \le H |x-z|,
\end{equation}  and
\begin{equation}
\label{eq-q-Lip-new}
\sum_{l\in \ss\backslash\{k\}} \abs{q_{kl}(x) - q_{kl}(y)} \le H \abs{x-y},   
\end{equation} where the constant $H>0$ is the same as in
Assumption~\ref{I1} without loss of generality.}\end{Assumption}

\begin{Remark}
In \eqref{Lip-condition} of Assumption \ref{I1}, we assumed that $\int_{U} |c(x,k,u) - c(y,k,u)|^{2}\Pi(\d u) \le H |x-y|^{2}$ for all $x,y\in \R^{d}$ and $k \in \ss$. This condition in general does not necessarily imply \eqref{eq-c0-nu-lip}. Consider for example $U= (0, 1)$ and $\Pi (\d u) = \frac{\d u}{u^{1+\alpha}}$ with some $\alpha\in (0,1)$. We can check directly that the function $c(x,k,u): = x u^{\frac{3}{4}\alpha}$ satisfies  \eqref{Lip-condition}  but not \eqref{eq-c0-nu-lip}.
\end{Remark}

The main result of   this section is:
\begin{Theorem}\label{thm-Feller}
Suppose that Assumptions \ref{I1}, \ref{qH}, and  \ref{qkappa} 
hold. Then the process $(X,\La)$ generated by the operator $\A$ of \eqref{eq-operator} has Feller property.
\end{Theorem}

We will use the coupling method to prove Theorem \ref{thm-Feller}. To this end, let us first construct a coupling operator $\wdt \A$ for $\A$.
  For $x,z\in \R^{d}  $ and $i,j\in \ss $, we set
$$a(x,i,z,j)=\begin{pmatrix}
a(x,i) & \sigma (x,i) \sigma (z,j)' \\
\sigma (z,j) \sigma (x,i)' & a(z,j)
\end{pmatrix},\quad
b(x,i,z,j)=\begin{pmatrix}
b(x,i)\\
b(z,j) \end{pmatrix},$$ where $a(x,i) = \sigma (x,i)\sigma (x,i)'$
 and $ a(z,j)$ is similarly defined. 
 Next, for $f(x,i, z,j) \in C_{c}^{2} (\R^{d} \times \ss\times \R^{d} \times \ss )$, we define
\begin{equation}\label{eq-Omega-d-defn}
\wdt {\Omega}_{\text{diffusion}}f(x,i,z,j)=\frac
{1}{2}\hbox{tr}\bigl(a(x,i,z,j)D^{2}f(x,i, z,j)\bigr)+\langle
b(x,i,z,j), D f(x,i, z,j)\rangle,
\end{equation} where in the above, $Df(x,i,z,j)$ represents the gradient of $f$ with respect to the variables $x$ and $z$, that is,  $Df(x,i,z,j) = (D_{x}f(x,i,z,j), D_{z}f(x,i,z,j))'$. Likewise,  $D^{2}f(x,i, z,j)$ denotes the Hessian of $f$ with respect to the variables $x$ and $z$.
Let us also define for $f(x,i, z,j) \in C^{2}_{c}(\R^{d}  \times\ss \times \R^{d} \times \ss)$,
\begin{align}\label{eq-Omega-j-defn}
\nonumber   & \displaystyle\wdt {\Omega}_{\text{jump}}  f(x,i, z,j)
\\ & \  = \int_{U}
[f(x+c(x,i,u),i, z+c(z,j,u), j)-f(x,i,z,j)  \\
\nonumber & \qquad -  \langle D_{x } f(x,i,z,j),   c(x,i,u)\rangle \one_{\{ u \in U_{0}\}} - \langle D_{z} f(x,i, z,j),   c(z,j,u)\rangle  \one_{\{ u \in U_{0}\}}] \Pi(\d u), \end{align}
 which is a coupling of the jump
part in the generator ${\LL}_{i}$ defined in \eqref{eq-Lk-operator-defn}.
Next we define  the basic coupling (see, e.g., p. 11 on  \cite{Chen04}) for the $q$-matrices $Q(x) $ and $Q(y)$. For any $f(x,i, z,j) \in C_{c}^{2} (\R^{d} \times \ss\times \R^{d} \times \ss )$, we define
\begin{equation}
\label{eq-Q(x)-coupling} \begin{aligned}
\wdt \Omega_{\text{switching}}  f(x,i,z,j): =
 &\sum_{l\in\ss}[q_{il}(x)-q_{jl}(z)]^+(   f(x,l, z, j)-  f(x,i, z, j))\\
&+\sum_{l\in\ss}[q_{jl}(z) -q_{il}(x)]^+( f(x,i, z, l)- f(x,i, z, j) )\\
&+\sum_{l\in\ss}[ q_{il}(x) \wedge q_{jl}(z) ](  f(x,l, z, l)-f(x,i, z, j)).
\end{aligned}\end{equation}
It is easy to verify that $\wdt Q (x,z) $ defined in \eqref{eq-Q(x)-coupling} is a coupling to $Q(x)$ defined in \eqref{eq-Qx-operator}.

 Finally, the coupling operator to $\A$ of \eqref{eq-operator} can be written as
\begin{equation}
\label{eq-A-coupling-operator} \begin{aligned}
\wdt{ \mathcal{A}} & f(x,i,z,j)  : =\! \bigl[ \wdt \Omega_{\text{diffusion}}   + \wdt \Omega_{\text{jump}}   + \wdt \Omega_{\text{switching}} \bigr] f(x,i,z,j).
\end{aligned}\end{equation}
In fact, we can verify directly that for any $f(x,i,z,j) = g(x,i) \in C^{2}_{c} (\R^{d}\times \ss)$, we have $\wdt \A f(x,i,z,j) = \A g(x,i)$.

As in the proof of Proposition 5.2.13 in \cite{Karatzas-S}, we can construct a sequence $\{ \psi_{n}(r)\}_{n=1}^{\infty}$ of twice continuously differentiable functions satisfying
  $\abs{\psi_{n}'(r)} \le 1$ and
  $\lim_{n\to\infty} \psi_{n}(r) = |r|$ for $r\in \R$, and $0\le \psi_{n}''(r) \le 2 n^{-1}H^{-1}r^{-2} $ for $r\not=0$, where $H $ is as in \eqref{Lip-condition}.
Furthermore, for every $r\in \R$, the sequence $\{\psi_{n}(r) \}_{n=1}^{\infty}$ is nondecreasing.
 \begin{Lemma}
 For each $n \in \mathbb N$, let the function $\psi_{n}$ be defined as above and further define the function   $$f_{n}(x,k, z,l): = \psi_{n}(|x-z|) + \one_{\{ k \neq l\}},  \ \  (x,k, z, l) \in \R^{d}\times \ss \times \R^{d}\times \ss.$$ Then for all $(x,k,z,k) \in \R^{d}\times \ss \times \R^{d}\times \ss$ with $x\neq z$, we have
 \begin{equation}
\label{eq-A-couple-est}
\wdt \A f_{n}(x,k,z,k) \le  \frac{1}{n} + C |x-z|,
\end{equation} in which 
$C= C(H)$ is a positive constant.
\end{Lemma}
\begin{proof}
 For any $x,z \in \R^{d}$ and $k,l\in \ss$, set 
\begin{align*}
A(x,k, z,l)& =a(x,k)+a(z,l)-2 \sigma (x,k) \sigma (z,l)',\\
\wdh {B}(x,k,z,l)& =\langle x-z, b(x,k)-b(z,l) \rangle, \end{align*}
 and
\begin{displaymath}
\overline{A}(x,k, z,l)   = \langle x-z, A(x,k, z,l) (x-z) \rangle / |x-z|^{2}.
\end{displaymath}
Then as in the proof of Theorem 3.1 in \cite{ChenLi-89}, we can verify  that
\begin{align*}
2\, \wdt \Omega_{\text{diffusion}}f_{n}(x,k, z,l)  & = \psi_{n}''(|x-z|)\overline A (x,k, z,l)  \\ & \qquad + \frac{ \psi_{n}'(|x-z|) }{|x-z|} \big[ \tr(A(x,k, z,l))-\overline A(x,k, z,l) + 2 \wdh B(x,k, z,l)\big].
\end{align*}
Note that
$\tr (  A(x,k,z,k))=\|\sigma (x,k)-\sigma (z,k)\|^{2}$ and hence we obtain from    \eqref{Lip-condition}  that  $$\tr A(x,k,z,k)  + \wdh B(x,k,z,k) \le H |x-z|^{2}.$$   On the other hand, using \eqref{Lip-condition} again,
\begin{displaymath}
\overline A(x,k, z,k) = \frac{\langle x-z, (\sigma(x,k)-\sigma(z,k)) (\sigma(x,k)-\sigma(z,k))^{T} (x-z)\rangle}{|x-z|^{2}} \le H |x-z|^{2}.
\end{displaymath}
Thus it follows that
\begin{equation}
\label{eq1-Omega-d}\begin{aligned}
\wdt \Omega_{\text{diffusion}}  f_{n}(x,k,z,k) &   \le \frac{1}{2} \psi_{n}''(|x-z|) H  |x-z|^{2} + \frac{3}{2} \psi_{n}'(|x-z|) H |x-z| \\
          & \le  \frac{1}{n} +  \frac{3}{2}   H |x-z|,
\end{aligned}\end{equation} where the last inequality follows from the construction of the function $\psi_{n}$.

Next we show that  for some positive constant $K$, we have
\begin{equation}
\label{eq-jump-est}
\wdt \Omega_{\text{jump}} f_{n}(x,k,z,k)  \le K     |x-z|.   
\end{equation}
In fact, since $|\psi_{n}'| \le 1$,  we can use 
 \eqref{eq-c0-nu-lip} to compute
\begin{align*}
\int_{U_{0}^{c}} &  [\psi_{n}(|x+ c(x,k,u) - z - c(z,k,u)|) - \psi_{n}(|x-z|)] \Pi(\d u)  \\
 & \le  \int_{U_{0}^{c}} |  c(x,k,u)  - c(z,k,u)| \Pi(\d u) 
   \le H |x-z|.
\end{align*}  
On the other hand, note that $D_{z} \psi_{n}(|x-z|) =- D_{x} \psi_{n}(|x-z|)$. Thus it follows that
 \begin{align*}
& \int_{U_{0}}   [\psi_{n}(|x+ c(x,k,u) - z - c(z,k,u)|) - \psi_{n}(|x-z|)  \\ & \qquad - \lan D_{x} \psi_{n}(|x-z|), c(x,k,u)\ran - \lan D_{z} \psi_{n}(|x-z|), c(z,k,u)\ran ] \Pi(\d u)    \\
 & \ \ = \int_{U_{0}}  [\psi_{n}(|x- z  + c(x,k,u) - c(z,k,u)|) - \psi_{n}(|x-z|)   \\ & \qquad \qquad
 - \lan D_{x} \psi_{n}(|x-z|), c(x,k,u)- c(z,k,u) \ran   ] \Pi(\d u) \\
 & \ \ \le 2  \int_{U_{0}} \abs{c(x,k,u)- c(z,k,u)} \Pi(\d u) \\
 & \ \ \le 2 H |x-z|,
\end{align*} where we used \eqref{eq-c0-nu-lip} to obtain the last inequality. Combining the above two displayed equations  gives \eqref{eq-jump-est}.

Finally we estimate $\wdt \Omega_{\text{switching}} f_{n}(x,k,z,l)$.  Clearly we have $\wdt \Omega_{\text{switching}} f_{n}(x,k,z,l) \le 0$ when $k \neq l$. When $k =l$, we have  from \eqref{eq-q-Lip-new} that
\begin{align} \label{eq1-switching-est}
\nonumber \wdt \Omega_{\text{switching}} f_{n}(x,k,z,k) & = \sum_{ i\in \ss} [q_{ki}(x)-q_{ki}(z)]^+( \one_{\{i \neq k \}} - \one_{\{ k\neq k\}})\\
\nonumber&\qquad +\sum_{i \in\ss}[q_{ki}(z)  -q_{ki}(x)]^+(  \one_{\{i \neq k \}} - \one_{\{ k\neq k\}} ) + 0 \\
\nonumber & \le \sum_{i \in\ss, i \neq k} \abs{q_{ki}(x)-q_{ki}(z)} \\ &
 \le H |x-z|.
\end{align}

Now plug   \eqref{eq1-Omega-d}, \eqref{eq-jump-est}, and \eqref{eq1-switching-est} into \eqref{eq-A-coupling-operator} yields \eqref{eq-A-couple-est}. This completes the proof. \end{proof}

\begin{proof}[Proof of Theorem  \ref{thm-Feller}] Denote by  $\{ P(t,x,k, A): t \ge 0, (x,k) \in \R^{d}\times \ss, A \in \B(\R^{d}\times \ss)\}$ the transition probability family of the process $(X,\La)$.
Since $\ss$ has a discrete topology, we need only to show that for each $t\ge 0$ and $k \in \ss$, $P(t,x,k,\cdot) $ converges weakly to $P(t,z,k,\cdot) $ as $x-z\to 0$. By virtue of Theorem 5.6 in \cite{Chen04}, it suffices to prove that \begin{equation}
\label{eq:Wasserstein-convergence}
W(P(t,x,k,\cdot), P(t,z,k,\cdot) ) \to 0 \text{ as } x\to z,
\end{equation}
where $W(\cdot, \cdot)$ denotes the Wasserstein metric between two probability measures.

 Let
$(\wdt X(t), \wdt \La(t),   \wdt Z(t), \wdt \Xi(t))$ denote the coupling process
corresponding to the coupling operator  $\wdt \A$ defined in \eqref{eq-A-coupling-operator}.  Assume that  $(\wdt X(0), \wdt \La(0),   \wdt Z(0), \wdt \Xi(0)) = (x,k,z,k)\in \R^{d}\times\ss \times \R^{d}\times \ss$ with $x\neq z$.
 Define $\zeta: = \inf\{ t \ge0: \wdt \La(t) \neq \wdt \Xi(t)\}$. Note that $\P\{ \zeta > 0\} =1$. In addition,
similarly to the proof of Theorem 2.3 in
\cite{ChenLi-89}, set
\begin{align*}
T_R &: = \inf \{t \ge 0: |\wdt {X}(t)|^2 +|\wdt {Z}(t)|^2 + \wdt \La(t) + \wdt \Xi(t) > R \}.
\end{align*}

Now we apply It\^o's formula to the process $f_{n}(\wdt X(\cdot), \wdt \La(\cdot),   \wdt Z(\cdot), \wdt \Xi(\cdot) )$   to obtain
\begin{equation}\label{eq-fn-mean}\begin{aligned}
\E&\bigl [ f_{n}(\wdt X(t\wedge T_{R} \wedge \zeta), \wdt \La(t\wedge T_{R} \wedge \zeta),   \wdt Z(t\wedge T_{R} \wedge \zeta), \wdt \Xi(t\wedge T_{R} \wedge \zeta) ) \bigr]\\ &  = f_{n}(x,k,z,k) + \E \biggl[ \int_{0}^{t\wedge T_{R}\wedge \zeta} \wdt \A f_{n} (\wdt X(s), \wdt \La(s),   \wdt Z(s), \wdt \Xi(s)  ) \d s  \biggr]\\
  & \le \psi_{n}( |x - z|)  + \frac{t}{n} + C \E \biggl[ \int_{0}^{t\wedge T_{R}\wedge \zeta} |\wdt X (s) - \wdt Z (s)| \d s \biggr],
\end{aligned}\end{equation}
where the last step follows from the observation  that $\wdt \La(s) = \wdt \Xi(s)$ for all $s\in [0, t\wedge T_{R}\wedge \zeta)$ and the estimate in  \eqref{eq-A-couple-est}.  Since $f_{n}(x,k,z,l) =\psi_{n}(|x-z|) + \one_{\{ k \neq l\}} \ge \psi_{n}(|x-z|) $, we have from \eqref{eq-fn-mean} that
\begin{align*}
\E&\bigl [ \psi_{n}(|\wdt X(t\wedge T_{R} \wedge \zeta)-   \wdt Z(t\wedge T_{R} \wedge \zeta)|  ) \bigr] \\ & \le \psi_{n}( |x - z|)  + \frac{t}{n} + C \E \biggl[ \int_{0}^{t\wedge T_{R}\wedge \zeta} |\wdt X (s) - \wdt Z (s)| \d s \biggr].
\end{align*}
Recall that  $\psi_{n}(|x|) \uparrow |x|$ as $n \to \infty$. Therefore, passing to the limit as $n \to \infty$ on both sides of the above equation, it follows from the Monotone Convergence Theorem that
\begin{align*}
\E & \bigl[|\wdt X(t\wedge T_{R} \wedge \zeta)-   \wdt Z(t\wedge T_{R} \wedge \zeta) |\bigr] \\ &  \le |x-z| +   C  \E \biggl[ \int_{0}^{t\wedge T_{R} \wedge \zeta}  |\wdt X (s) - \wdt Z (s)| \d s \biggr] \\
 & = |x-z| + C  \E \biggl[ \int_{0}^{t}  |\wdt X (s\wedge T_{R} \wedge \zeta) - \wdt Z (s\wedge T_{R} \wedge \zeta)|   \d s \biggr]  .
\end{align*}
 Then an application of   Gronwall's inequality leads to
$$\E \bigl[\bigl| \wdt X(t\wedge T_{R} \wedge \zeta)-   \wdt Z(t\wedge T_{R} \wedge \zeta)  \bigr| \bigr] \le |x-z| \exp (C t).$$
Now passing to the limit as   $R \uparrow \infty$, 
we conclude that
\begin{equation}\label{eq:diff-before-switching-mean} \E \bigl[\bigl| \wdt X(t \wedge \zeta)-   \wdt Z(t \wedge \zeta)  \bigr| \bigr] \le |x-z| \exp (C t).
 \end{equation}

 Observe that $\zeta \le t$ if and only if $\wdt \La(t \wedge \zeta) \neq \wdt\Xi(t\wedge \zeta).$ Put $f(x,k,z,l): = \one_{\{ k\neq l\}}$ and apply It\^o's formula to the process $f(\wdt X(t),\wdt \La(t),\wdt Z(t),\wdt\Xi(t))$:
 \begin{align}\label{eq:zeta<t-prob}
\nonumber \P\{\zeta \le t \}& = \E[\one_{\{ \wdt \La(t \wedge \zeta) \neq \wdt\Xi(t\wedge \zeta)\}}] = \E[f(\wdt X(t\wedge \zeta),\wdt \La(t\wedge \zeta),\wdt Z(t\wedge \zeta),\wdt\Xi(t\wedge \zeta))] \\
\nonumber & = \E\biggl[\int_{0}^{t\wedge \zeta} \wdt \A f(\wdt X(s),\wdt \La(s),\wdt Z(s),\wdt\Xi(s)) \d s \biggr]\\
 \nonumber& \le H \E  \biggl[\int_{0}^{t\wedge \zeta} |\wdt X(s) - \wdt Z(s)| \d s\biggr]   \\
 \nonumber& =  H  \int_{0}^{t}\E[ |\wdt X(s\wedge \zeta) - \wdt Z(s\wedge \zeta)| ]\d s \\
 & \le K |x-z| e^{Ct},
\end{align}
where $K= K(H, \Pi(U_{0}^{c}))$ is a positive constant,  the first inequality above follows from  \eqref{eq1-switching-est} and the last step follows from \eqref{eq:diff-before-switching-mean}.

The standard argument using Assumptions \ref{I1} and \ref{qH} reveals that $\E[\sup_{0\le s \le t} |\wdt X(s)|^{2} + | \wdt Z(s)|^{2}] \le K(1+|x|^{2} + |z|^{2})$, where $K= K(t, H, \Pi(U_{0}^{c})) $ is a positive constant.
 Then it follows from the H\"older inequality and \eqref{eq:zeta<t-prob} that
 \begin{equation}\label{eq:diff-after-switching-mean}
\E[|\wdt X(t) - \wdt Z(t) - \wdt X(t\wedge \zeta) + \wdt Z(t\wedge \zeta)| \one_{\{\zeta \le t \}}] \le K  (1+|x|^{2} + |z|^{2})^{\frac12}|x-z|^{\frac12},
\end{equation} where in the above, $K$ is a positive constant depending only on $t,H,$ and $\Pi(U_{0}^{c})$. Finally, we combine \eqref{eq:diff-before-switching-mean} and \eqref{eq:diff-after-switching-mean} to obtain
\begin{align}
\label{eq-diff-mean}
\nonumber\E& [|\wdt X(t) - \wdt Z(t)|] \\ &  \le \E[\bigl| \wdt X(t \wedge \zeta)-   \wdt Z(t \wedge \zeta)  \bigr| ] + \E[|\wdt X(t) - \wdt Z(t) - \wdt X(t\wedge \zeta) + \wdt Z(t\wedge \zeta)| \one_{\{\zeta \le t \}}]  \\
\nonumber  & \le K |x-z| +K  (1+|x|^{2} + |z|^{2})^{\frac12}|x-z|^{\frac12}.
 \end{align}
 Observe that if $\wdt \La(t) \neq \wdt\Xi(t)$ then $\zeta \le t$. Thus thanks to  \eqref{eq:zeta<t-prob}, we also have
\begin{equation}
\label{eq:La-Xi-difference prob}
 \E[\one_{\{\wdt \La(t) \neq \wdt\Xi(t) \}}] \le \P\{\zeta \le t \} \le  K |x-z| e^{Ct}.
\end{equation}

Now let $f\in C_{b}(\R^{d}\times \ss)$, then we have
\begin{equation}\label{eq:f-difference-mean}\begin{aligned}
\E& \big[|f(\wdt X(t), \wdt \La(t)) - f(\wdt Z(t), \wdt \Xi(t)) |\big]  \\
  & \le \E\big[|f(\wdt X(t), \wdt \La(t)) - f(\wdt Z(t), \wdt \La(t)) |\big]  + \E\big[|f(\wdt Z(t), \wdt \La(t)) - f(\wdt Z(t), \wdt \Xi(t)) |\big].  \end{aligned}
\end{equation} Both terms on the right-hand side of \eqref{eq:f-difference-mean} converge to 0 as $x\to z$ thanks to \eqref{eq-diff-mean}, \eqref{eq:La-Xi-difference prob}, the continuity of $f$,  and the bounded convergence theorem. This implies \eqref{eq:Wasserstein-convergence} and therefore completes the proof.
 \end{proof}


\section{Strong Feller Property: Jump Diffusions}\label{sect-StFP-jump-diffusion}

In order to prove the strong Feller property, we further make the following assumption:

\begin{Assumption}   \label{StFP1}{\rm Assume that
the characteristic measure $\Pi(\cdot)$ is finite (i.e.,
$U_{0}\equiv \emptyset$) and that for each $k \in \ss$, the
diffusion $X^{(k),0}$ satisfying
\begin{equation}\label{Xk0} \d X^{(k),0}(t)= b(X^{(k),0}(t),k)\d t+\sg
(X^{(k),0}(t),k)\d B(t),\end{equation} has the strong Feller
property and has a transition probability density
with respect to the Lebesgue measure.}\end{Assumption}

\begin{Remark}   \label{StFP2}{\rm
For a given $k \in \ss$, a 
 sufficient condition for
$X^{(k),0}$ to have the strong Feller property and to have a
transition probability density is that the Fisk-Stratonovich type
generator of $X^{(k),0}$ is hypoelliptic (see, for example,
\cite{IchiK-74, Klie-87} for  details). In particular, if the
diffusion matrix of $X^{(k),0}$ is uniformly positive, then
the diffusion process $X^{(k),0}$ must have the strong Feller
property and must have a transition probability density (see the
last paragraph of Section 2 in \cite{Klie-83} or Section 8 of Chapter
V in \cite{IkedaW-89}).}\end{Remark}

For later use, we now introduce a family of jump diffusions under Assumption~\ref{StFP1}.
For each $k \in \ss$, let the single jump diffusion $X^{(k)}$
satisfy the following stochastic differential-integral equation:
\begin{equation}\label{(EU1)}
 \d X^{(k)}(t)  =
b(X^{(k)}(t),k)\d t+\sigma (X^{(k)}(t),k)\d B(t) 
+\int_{U} c(X^{(k)}(t-),k,u)N(\d t,\d u).
\end{equation}

\begin{Lemma}   \label{StFP3}
Suppose that Assumption~\ref{StFP1} holds. For each
given $k \in \ss$, the jump-diffusion process $X^{(k)}$ has the
strong Feller property
with a transition probability density
with respect to   the Lebesgue measure.
\end{Lemma}

\begin{proof}
For a given $k \in \ss$, let us denote by $P^{(k)}(t,x,A)$ the
transition probability for the process $X^{(k)}$, and by
$P^{(k),0}(t,x,A)$ the transition probability for the process
$X^{(k),0}$. Following the proofs of \cite[Theorem 14 in Chapter
I]{Skorohod-89} and \cite[Lemma 2.3]{LiDS-02} with some elementary
analysis, for any given $t>0$, $x\in \R^{d}$ and $A \in
\B(\R^{d})$, we obtain the relation
\begin{equation}\label{(StFP2)}
\begin{array}{ll}
 P^{(k)}(t,x,A)\ad =\exp\{-t\Pi (U)\}P^{(k),0}(t,x,A)\\
\aad \quad
+\int_{0}^{t}\int\int_{U}\exp\{-s_{1}\Pi(U)\}P^{(k),0}(s_{1},x,\d y_{1})
\Pi(du_{1})\d s_{1}\\
\aad \qquad \qquad \qquad \times
P^{(k)}(t-s_{1},y_{1}+c(y_{1},k,u_{1}),A).
\end{array}
\end{equation} From this we have
\begin{equation}\label{(StFP3)}
\begin{array}{ll}
\ad P^{(k)}(t-s_{1},y_{1}+c(y_{1},k,u_{1}),A)\\
\aad \quad
=\exp\{-(t-s_{1})\Pi(U)\}P^{(k),0}(t-s_{1},y_{1}+c(y_{1},k,u_{1}),A)\\
\aad \qquad +\int_{0}^{t-s_{1}}\int\int_{U}\exp\{-s_{2}\Pi(U)\}
P^{(k),0}(s_{2},y_{1}+c(y_{1},k,u_{1}),\d y_{2})\\
\aad \qquad \qquad \qquad \qquad \times \Pi(\d u_{2})\d s_{2}
P^{(k)}(t-s_{1}-s_{2},y_{2}+c(y_{2},k,u_{2}),A).
\end{array}
\end{equation} Using (\ref{(StFP2)}) again we further have
\begin{equation}\label{(StFP4)}
\begin{array}{ll}
\ad P^{(k)}(t-s_{1}-s_{2},y_{2}+c(y_{2},k,u_{2}),A)\\
\aad \quad
=\exp\{-(t-s_{1}-s_{2})\Pi(U)\}P^{(k),0}(t-s_{1}-s_{2},y_{2}+c(y_{2},k,u_{2}),A)\\
\aad \qquad
+\int_{0}^{t-s_{1}-s_{2}}\int\int_{U}\exp\{-s_{3}\Pi(U)\}\\
\aad \qquad \qquad \times
P^{(k),0}(s_{3},y_{2}+c(y_{2},k,u_{2}),\d y_{3})
\Pi(\d u_{3})\d s_{3}\\
\aad \qquad \qquad \times
P^{(k)}(t-s_{1}-s_{2}-s_{3},y_{3}+c(y_{3},k,u_{3}),A).
\end{array}
\end{equation} Using (\ref{(StFP2)}) countably many times, we
conclude that for any given $t>0$, $x\in \R^{d}$ and $A \in
\B(\R^{d})$,
\begin{equation}\label{(StFP5)}
P^{(k)}(t,x,A)= \hbox{a series}.
\end{equation} For this series, from (\ref{(StFP2)})--(\ref{(StFP4)}) we
derive that the first term (in which the process has no jump on
$[0,t]$) is
\begin{equation}\label{(StFP6)}
\exp\{-t\Pi (U)\}P^{(k),0}(t,x,A),
\end{equation} the second term
(in which the process has just one jump on
$[0,t]$) is
\begin{equation}\label{(StFP7)}
\begin{array}{ll}
\ad \exp\{-t\Pi(U)\}\int_{0}^{t}\int\int_{U}
P^{(k),0}(s_{1},x,\d y_{1})
\Pi(\d u_{1})\d s_{1}\\
\aad \quad \times P^{(k),0}(t-s_{1},y_{1}+c(y_{1},k,u_{1}),A),
\end{array}
\end{equation} the third term (in which the process has just
two
jumps on
$[0,t]$) is
\begin{equation}\label{(StFP8)}
\begin{array}{ll}
\ad
\exp\{-t\Pi(U)\}\int_{0}^{t}\int\int_{U}\int_{0}^{t-s_{1}}\int\int_{U}
P^{(k),0}(s_{1},x,\d y_{1})
\Pi(\d u_{1})\d s_{1}\\
\aad \quad \times P^{(k),0}(s_{2},y_{1}+c(y_{1},k,u_{1}),\d y_{2})
\Pi(\d u_{2})\d s_{2}\\
\aad \quad \times P^{(k),0}(t-s_{1}-s_{2},y_{2}+c(y_{2},k,u_{2}),A),
\end{array}
\end{equation} and moreover, the general term (in which the process
has just $n$ jumps on
$[0,t]$) is
\begin{equation}\label{(StFP9)}
\begin{array}{ll}
\ad
\exp\{-t\Pi(U)\}\int_{0}^{t}\int\int_{U}\int_{0}^{t-s_{1}}\int\int_{U}
\cdots \int_{0}^{t-s_{1}-\cdots -s_{n-1}}\int\int_{U}\\
\aad \qquad \quad P^{(k),0}(s_{1},x,\d y_{1})
\Pi(\d u_{1})\d s_{1} \\
\aad \qquad \times P^{(k),0}(s_{2},y_{1}+c(y_{1},k,u_{1}),\d y_{2})
\Pi(\d u_{2})\d s_{2}\cdots \\
\aad \qquad \times
P^{(k),0}(s_{n},y_{n-1}+c(y_{n-1},k,u_{n-1}),\d y_{n})
\Pi(\d u_{n})\d s_{n}\\
\aad \qquad \times
P^{(k),0}(t-s_{1}-\cdots-s_{n},y_{n}+c(y_{n},k,u_{n}),A).
\end{array}
\end{equation} In general, it is easy to see that the $n$th term does not exceed
$$\frac{\bigl(t\Pi(U)\bigr)^{n-1}}{(n-1)!}\exp\{-t\Pi(U)\}.$$ Hence it follows that the series
in (\ref{(StFP5)}) converges uniformly with respect to $x$ over
$\R^{d}$. 

It is easy to prove that for any given $t>0$ and $A \in
\B(\R^{d})$, each term of the series in (\ref{(StFP5)}) is lower
semicontinuous with respect to $x$ by the strong Feller property of
$X^{(k),0}$ (see Assumption~\ref{StFP1}). Therefore, it follows that
for any given $t>0$ and $A \in \B(\R^{d})$, $P^{(k)}(t,x,A)$
is also lower semicontinuous with respect to $x$. As a result, $X^{(k)}$
has the strong Feller property by  Proposition
6.1.1 in  \cite{MeynT-93}. Finally, from (\ref{(StFP5)}),
$X^{(k)}$ has a
transition probability density with respect to
the Lebesgue measure
since
$X^{(k),0}$ does so under Assumption~\ref{StFP1}. The proof is complete. \end{proof}

\begin{Remark}   \label{StFP3a}{\rm
From (\ref{(StFP5)}) we can also see that if transition probability
density of $\wdt{X}^{(k),0}$ is positive, so is that of
$\wdt{X}^{(k)}$.}\end{Remark}

\section{Strong Feller Property: Regime-Switching Jump Diffusions}\label{sect-StFP}

 In order to prove the strong Feller property for $(X,\La)$, we further make the following assumption.

 \begin{Assumption}\label{assumption-finite-range}
There exists a positive integer $\kappa$ such that $q_{kl}(x)=0$ for all $k,l \in \ss$ with $|k-l|\ge \kappa+1$.
\end{Assumption}

Now let us establish the strong Feller property for the regime-switching jump diffusion $(X,\La)$.

\begin{Theorem}   \label{thm:str-Feller}
Suppose that Assumptions \ref{I1},  \ref{qH}, \ref{qkappa}, \ref{StFP1}, and  \ref{assumption-finite-range}  hold. Then
$(X,\La)$ has the strong Feller property.
\end{Theorem}

To proceed, we first consider the strong Feller property for a special type of switching jump-diffusion $(V,\psi)$. Let
the first component $V$ satisfy \begin{equation}\label{(GFP6)}\barray \d V(t)\ad =
b(V(t),\psi(t))\d t+\sg(V(t),\psi(t))\d B(t)\\
\aad\quad +\int_{U} c(V(t-),\psi(t-),u)N(\d t,\d u), \earray\end{equation} and the
second component $\psi$ that is independent of the Brownian motion
$B(\cdot)$ and Poisson random measure $N(\cdot,\cdot)$, be a
time-homogeneous Markov chain with state space $\ss$ satisfying
\begin{equation}\label{(GFP7)} \P  \{\psi(t+\Delta)=l | \psi(t)=k\} =\left \{ \barray
\wdh{q}_{kl}\Delta +o(\Delta),
\ad \, \, \hbox{if}\, \, k \ne l, \\
1+\wdh{q}_{kk}\Delta +o(\Delta), \ad \, \, \hbox{if}\, \, k = l
\end{array} \right.\end{equation} provided $\Delta \downarrow 0$, where $\wdh{Q}=\bigl(\wdh{q}_{kl}\bigr)$ is a conservative Q-matrix such that
(i) all the diagonal elements are equal to $-2 \kappa$, (ii) there are exactly $2\kappa$ off diagonal elements being 1 that are as symmetric and adjacent to the diagonal entry as possible, and (iii) all other elements are zero.
To be precise,
\begin{equation}
\label{eq-qkl}
\wdh{q}_{kl}= \begin{cases}
  -2\kappa    & \text{ if }k =l =0, 1,2,\dots, \\
  1    & \text{ if } k =0, 1, 2, \dots, \kappa-1,  \text{ and }  l =0, 1, 2, \dots, 2\kappa  \text{ with } l \neq k, \\
  1  &\text{ if } k =\kappa+1, \kappa+2, \dots, \text{ and } |l-k| \le \kappa, \\
  0 & \text{ otherwise}.
\end{cases}
\end{equation}
For example,  when $\kappa=1$,
\begin{equation*}\label{S}
\wdh{Q}=\bigl(\wdh{q}_{kl}\bigr)=\left(\begin{array}{ccccccc}
{  -2} & {  1} & {  1} & 0 & 0 & 0 & \cdots\\
{  1} & {  -2} & {  1} & 0 & 0 & 0 & \cdots\\
0 & {  1} & {  -2} & {  1} & 0 & 0 & \cdots\\
0 & 0 & {  1} & {  -2} & {  1} & 0 & \cdots\\
\vdots & \vdots & \vdots & \vdots & {  \vdots} & {  \vdots} & {  \ddots} 
\end{array} \right).
\end{equation*} Obviously, if the $-2$, $1$ and $1$ on the first row of this matrix were replaced by $-1$, $1$ and $0$,
then this matrix would be a very simple birth-death matrix. In the
sequel, we sometimes emphasize the process $(V(t),\psi(t))$ with
initial condition $(V(0),\psi(0))=(x,k)$ by
$(V^{(x,k)}(t),\psi^{(k)}(t))$. Moreover, denote by
$\Gamma(t,(x,k),\cdot)$ the transition probability of $(V,\psi)$.
For subsequent use, let us fix a probability measure $\mu (\cdot)$
that is equivalent to the product measure on $\R^{d} \times \ss$ of
the Lebesgue measure on $\R^{d}$ and the counting measure on $\ss$.
For example, $\mu (\cdot)$ could be taken as the product measure of
the Gaussian probability measure on $\R^{d}$ and the Poisson
probability measure on $\ss$.

\begin{Lemma}   \label{thm:GFP4}
Suppose that Assumptions \ref{I1},  \ref{StFP1}, and    \ref{assumption-finite-range} hold. Then
$(V,\psi)$ has the strong Feller property and the transition
probability $\Gamma(t,(x,k),\cdot)$ of $(V,\psi)$ has density
$\gamma(t,(x,k),\cdot)$
with respect to
 $\mu (\cdot)$.
\end{Lemma}

\begin{proof}
Denote by the $\upsilon_{1}$ the stopping time defined by
$\upsilon_{1}=\inf \{s>0: \psi(t) \ne \psi(0)\}$. When $\psi(0)=k$,
$(\upsilon_{1},\psi(\upsilon_{1}))$ on $[0,\infty)\times
\ss_{k}$ with respect to the product of the Lebesgue
measure and the counting measure has the probability density
$\exp\bigl(-2\kappa s\bigr){\mathbf{1}}_{\ss_{k}}(l)$,
where $\ss_{k}:= \{l \in \ss: \wdh q_{kl} =1 \}$ 
is a finite subset of $\ss$.
For any given $t>0$, $x\in \R^{d}$, $k,l\in \ss$ and $A \in
\B(\R^{d})$, we have the relation \begin{equation}\label{(GFP10)} \begin{aligned}
\Gamma \bigl(t,(x,k),A \times
\{l\}\bigr) & =\delta_{kl}\exp\{-2\kappa t\}P^{(k)}(t,x,A)
   \\ &\ \ +\int_{0}^{t}\!\sum_{l_{1} \in \ss_{k}}\! \int
\exp\{-2\kappa s_{1}\}P^{(k)}(s_{1},x,\d y_{1}) 
\Gamma(t-s_{1},(y_{1},l_{1}),A \times
\{l\})\d s_{1}, \end{aligned}\end{equation} where $\delta_{kl}$ is the Kronecker
symbol in $k$, $l$, which equals $1$ if $k=l$ and is $0$ if $k\neq
l$. From this we have \begin{equation}\label{(GFP10a)} \barray
\Gamma\bigl(t-s_{1},(y_{1},l_{1}),A \times
\{l\}\bigr)\ad
=\delta_{l_{1}l}\exp\{-2\kappa (t-s_{1})\}
P^{(l_{1})}(t-s_{1},y_{1},A)\\
\aad \quad +\int_{0}^{t-s_{1}}\sum_{l_{2}\in \ss_{l_{1}}} \int
\exp\{-2\kappa s_{2}\}P^{(l_{1})}(s_{2},y_{1},\d y_{2})\\
\aad \qquad \qquad \qquad \times \Gamma(t-s_{1}-s_{2},(y_{2},l_{2}),A
\times \{l\})\d s_{2}. \earray\end{equation} Using (\ref{(GFP10)}) countably
many times, as in the proof of Lemma \ref{StFP3}, we conclude that
for any given $t>0$, $x\in \R^{d}$ and $A \in \B(\R^{d})$,
\begin{equation}\label{(GFP10b)} \Gamma\bigl(t,(x,k),A \times \{l\}\bigr)= \hbox{a series}.
\end{equation} For this series, as in the proof of Lemma \ref{StFP3}, we
derive that the first term (in which $\psi$ has no
jump on $[0,t]$)
is
\begin{equation}\label{(GFP10c)}
\delta_{kl}\exp\{-2\kappa t\}P^{(k)}(t,x,A),
\end{equation} the second term (in which $\psi$ has just one
jump on
$[0,t]$) is
\begin{equation}\label{(GFP10d)}
\exp\{-2\kappa t\}\int_{0}^{t} \sum_{l_{1}\in \ss_{k}, l_{1}=l} \int
P^{(k)}(s_{1},x,\d y_{1}) P^{(l_{1})}(t-s_{1},y_{1},A)\d s_{1},
\end{equation} and the third term (in which $\psi$ has just
two jumps on
$[0,t]$) is 
\begin{equation}\label{(GFP10e)}
\begin{array}{ll}
\ad \exp\{-2\kappa t\} \int_{0}^{t}\int_{0}^{t-s_{1}}
\sum_{l_{1}\in \ss_{k}, l_{2}\in \ss_{l_{1}}, l_{2}=l} \int \int
P^{(k)}(s_{1},x,\d y_{1})\\
\aad \qquad\ \  \times P^{(l_{1})}(s_{2},y_{1},\d y_{2})
P^{(l_{2})}(t-s_{1}-s_{2},y_{2},A)\d s_{2}\d s_{1}.
\end{array}
\end{equation} Similar to the proof of Lemma of \ref{StFP3}, we can easily verify that the $n$th term of the series in \eqref{(GFP10b)}  is bounded above by $\frac{(2\kappa t)^{n-1}}{(n-1)!}\exp\{-2\kappa t\}.$ Thus it is uniformly convergent with respect to $x\in \R^{d}$.  Noting
that $\ss$ is a infinitely countable set
with a
discrete metric, and using similar arguments as those in the
proof of Lemma \ref{StFP3}, we derive Lemma \ref{thm:GFP4}.
\end{proof}


\begin{Lemma}   \label{FP2}
Suppose that Assumptions \ref{I1},   \ref{StFP1}, and  \ref{assumption-finite-range} hold. Then for all $T>0$, $\delta
>0$ and $k \in \ss$, we have
\begin{equation}\label{(FP4)}\P  \biggl\{\sup_{0 \le t \le T}
\bigl|V^{(x,k)}(t)-V^{(y,k)}(t)\bigr| \ge \delta \biggr\} \to 0
\end{equation}
as $|x-y| \to 0$.
\end{Lemma}

\begin{proof}
This lemma is just  \cite[Lemma 4.1]{Xi-09}.
\end{proof}

\begin{Lemma}   \label{lem:GFP5}
For any  bounded and measurable function $f$ on $\R^{d} \times \ss$
and any  positive number $\delta >0$, there exists a compact subset
$D \subset \R^{d}$ such that $\mu (D^{c} \times \ss) <\delta$ and
$f|_{D \times \ss}$, the function $f$ restricted to $D \times \ss$,
is uniformly continuous.
\end{Lemma}

\begin{proof}
This lemma can be derived from the Lusin Theorem (see, for example
\cite[Theorem 7.4.3]{Cohn-80}).
\end{proof}

\begin{Lemma}   \label{lem:GFP6}
Suppose that Assumptions \ref{I1},   \ref{StFP1}, and  \ref{assumption-finite-range} hold. For any
given $t>0$ and bounded measurable function $f$ on $\R^{d} \times
\ss$, we have that
\begin{equation}\label{(GFP11)}
f(V^{(x,k)}(t),\psi^{(k)} (t)) \to f(V^{(y,k)}(t),\psi^{(k)} (t))
\quad \hbox{in probability}
\end{equation}
as $|x-y| \to 0$.
\end{Lemma}

\begin{proof}
It follows from Lemma \ref{thm:GFP4} that for any $(x,k) \in \R^{d}
\times \ss$, any $A \in {\B}(\R^{d})$ and $l \in \ss$,
\begin{equation}\label{(GFP12)}
\P  \bigl((V^{(x,k)}(t),\psi^{(k)} (t)) \in A \times
\{l\}\bigr)=\int_{A} \gamma(t,(x,k),(y,l)) \mu(\d y \times \{l\}).
\end{equation} By the strong Feller property
proved in Lemma \ref{thm:GFP4}, for any sequence $\{x_{n}\}$
satisfying $x_n\to x$
and for any $g(y,l) \in L^{\infty} (\mu)$, we have
$$\sum_{l \in \ss} \int g(y,l)\gamma(t,(x_{n},k),(y,l))
\mu(\d y \times \{l\}) \to \sum_{l \in \ss} \int
g(y,l)\gamma(t,(x,k),(y,l))\mu(\d y \times \{l\})$$ as $n \to \infty$.
Namely, when $n \to \infty$, $\gamma(t,(x_{n},k),\cdot)$ converges
weakly to $\gamma(t,(x,k),\cdot)$ in $L_{1}(\mu)$.
Thus, by the Dunford-Pettis theorem,
we obtain that the family
$\{\gamma(t,(x_{n},k),\cdot): n \ge 1\}$ is uniformly integrable in
$L_{1}(\mu)$. Hence for any given $\varepsilon>0$, there exists
a $\delta >0$ such that for all $A \in {\B}(\R^{d})$, if $\mu (A
\times \ss) <\delta$, then for all $n \ge 1$,
\begin{equation}\label{(GFP13)}
\P  \bigl((V^{(x_n,k)}(t),\psi^{(k)} (t)) \in A \times \ss
\bigr)=\sum_{l \in \ss} \int_{A}\gamma(t,(x_{n},k),(y,l))\mu(\d y \times
\{l\})<\varepsilon,
\end{equation}
and
\begin{equation}\label{(GFP14)}
\P  \bigl((V^{(x,k)}(t),\psi^{(k)} (t)) \in A \times \ss \bigr)
=\sum_{l \in \ss} \int_{A}\gamma(t,(x,k),(y,l))\mu(\d y \times
\{l\})<\varepsilon.
\end{equation}

By Lemma \ref{lem:GFP5}, we find a compact subset $D \subset
\R^{d}$ such that $\mu (D^{c} \times \ss) <\delta$ and $f|_{D
\times \ss}$ is uniformly continuous. Namely, for any given $\eta
>0$, there exists $\delta_{1}>0$ such that for all $(x,k),\,
(x^{\prime},k) \in D \times \ss$, if $|x-x^{\prime}|<\delta_{1}$,
then $|f(x,k)-f(x^{\prime},k)|<\eta$ for all $k \in \ss$. Therefore,
from (\ref{(GFP13)}) and (\ref{(GFP14)}), we arrive at
\begin{equation}\label{(GFP15)}
\begin{aligned} \displaystyle
\P &  \bigl(|f(V^{(x_n,k)}(t),\psi^{(k)}(t))
-f(V^{(x,k)}(t),\psi^{(k)}(t))|>\eta \bigr)\\ &
  \le
\P    \bigl(|V^{(x_n,k)}(t)-V^{(x,k)}(t)|>\delta_{1} \bigr)\\
& \quad \quad \displaystyle +\P  \bigl((V^{(x_n,k)}(t),\psi^{(k)}(t))
\notin D \times \ss \bigr)
+\P  \bigl((V^{(x,k)}(t),\psi^{(k)}(t)) \notin D \times \ss \bigr)\\
&  \le \P
\bigl(|V^{(x_n,k)}(t)-V^{(x,k)}(t)|>\delta_{1} \bigr)+2\varepsilon.
\end{aligned}
\end{equation}
Meanwhile, by Lemma \ref{FP2}, $\P  \bigl(|V^{(x_n,k)}(t)-V^{(x,k)}(t)|>\delta_{1} \bigr) \to 0$ as $n \to \infty$.
Inserting this into (\ref{(GFP15)}) and noting that $\varepsilon$ and $\eta$ are arbitrary, (\ref{(GFP11)}) holds. This
completes the proof.  \end{proof}

In order to transfer the strong Feller property from $(V,\psi)$ to
$(X,\La)$, we need to make a comparison between these two processes.
Let $\{\upsilon_{m}\}$ be the sequence of stopping times defined by
$$\upsilon_{0}=0, \qquad
\upsilon_{m+1}=\inf \{s> \upsilon_{m}: \psi(t) \ne \psi(\upsilon_{m})\} \ \hbox{ for }
\ m\ge 0 .$$ Define $n(t)=\max \{m: \upsilon_{m} \le t
\}$, which is the number of switches (i.e., jumps) of $\psi$ up  to  time $t$. Set $D:=D([0, \infty), \R^{d} \times \ss)$
and denote by ${\mathcal D}$ the usual $\sigma$-field of $D$. Likewise, for any $T>0$, set $D_{T}:=D([0, T], \R^{d} \times
\ss)$ and denote by ${\mathcal D}_{T}$ the usual $\sigma$-field of $D_{T}$. Moreover, denote by $\mu_{1} (\cdot)$ the
probability distribution induced by $(X,\La)$ and $\mu_{2} (\cdot)$ the probability distribution induced by $(V,\psi)$
in the path space $\bigl(D, {\mathcal D} \bigr)$, respectively. Denote  by $\mu_{1}^{T} (\cdot)$ the restriction of $\mu_{1}
(\cdot)$ and $\mu_{2}^{T} (\cdot)$ the restriction of $\mu_{2} (\cdot)$ to $\bigl(D_{T}, {\mathcal D}_{T} \bigr)$,
respectively. For any given $T>0$, from \cite[Lemma 4.2]{Xi-09}, we know that $\mu_{1}^{T} (\cdot)$ is absolutely
continuous with respect to $\mu_{2}^{T} (\cdot)$ and the corresponding Radon-Nikodym derivative has the following form.
\begin{equation}\label{(GFP8)}
\begin{array}{ll}
M_{T}\bigl(V(\cdot),\psi (\cdot) \bigr) \ad := \frac{\d \mu_{1}^{T}}{\d
\mu_{2}^{T}} \bigl(V(\cdot),\psi (\cdot) \bigr)\\
\aad = \prod_{i=0}^{n(T)-1}q_{\psi(\upsilon_{i})\psi(\upsilon_{i+1})}
\bigl(V(\upsilon_{i+1})\bigr)
 \exp \biggl(-\sum_{i=0}^{n(T)}
\int_{\upsilon_{i}}^{\upsilon_{i+1}\wedge T}
\bigl[q_{\psi(\upsilon_{i})}(V(s))-2\kappa \bigr] \d s \biggr),
\end{array}\end{equation} where
$q_{k}(x)= \sum_{l \neq k}q_{kl}(x)$.

\begin{Remark}
 Note that the Radon-Nikodym derivative defined in \eqref{(GFP8)} is similar to the likelihood ratio martingale defined in \cite{ChowT-97} and \cite{RogersW-V1}.
\end{Remark}

We restate \cite[Lemmas 4.3 and 4.4]{Xi-09} as the following two lemmas
respectively.

\begin{Lemma}   \label{4.3}
For all $T>0$, we have that
\begin{equation}\label{(4.21)}
\E\left[ \left|M_{T}\bigl(V^{(x,k)}(\cdot),\psi^{(k)}(\cdot) \bigr)-M_{T}\bigl(V^{(y,k)}(\cdot),\psi^{(k)}(\cdot)
\bigr)\right| \right]\to 0
\end{equation}
as $|x-y| \to 0$.
\end{Lemma}

\begin{Lemma}   \label{4.4}
For all $T>0$ and $(x,k) \in \R^{d} \times \ss$, $M_{T}\bigl(V^{(x,k)}(\cdot),\psi^{(k)} (\cdot) \bigr)$ is integrable.
\end{Lemma}

Now we are ready to prove the main result of this section.

\begin{proof}[Proof of Theorem  \ref{thm:str-Feller}] To prove the desired strong
Feller property, it is enough to prove that for any $t>0$ and any
bounded measurable function $f$ on $\R^{d} \times \ss$, $\E[
f(X^{(x,k)}(t),\La^{(x,k)}(t))]$ is bounded continuous in both $x$
and $k$. Since $\ss$ has a discrete metric, it is
sufficient to prove that
\begin{equation}\label{(GFP16)}
\abs{\E[ f(X^{(x,k)}(t),\La^{(x,k)}(t))]- \E
[f(X^{(y,k)}(t),\La^{(y,k)}(t))]} \to 0
\end{equation} as $|x-y| \to 0$. Indeed, by (\ref{(GFP8)}), for all $(x,k) \in \R^{d} \times \ss$,
\begin{equation}\label{(4.24)} \E[ f(X^{(x,k)}(t),\La^{(x,k)}(t))] =\E\big[
f(V^{(x,k)}(t),\psi^{(k)}(t))\cdot
M_{t}\bigl(V^{(x,k)}(\cdot),\psi^{(k)}(\cdot) \bigr)\bigr]. \end{equation} Similarly
to the proof of Proposition 1.2 in \cite{Wu-01}, for any given
$\varepsilon >0$, using (\ref{(4.24)}), we have \begin{equation}\label{(4.26)}
\begin{aligned} \displaystyle
& \bigl|\E[ f(X^{(x,k)}(t),\La^{(x,k)}(t))]-\E[ f(X^{(y,k)}(t),\La^{(y,k)}(t))]\bigr|\\ &
\quad \le \displaystyle \E
\Bigl[\Bigl|f(V^{(x,k)}(t),\psi^{(k)}(t))\cdot
M_{t}\bigl(V^{(x,k)}(\cdot),\psi^{(k)} (\cdot) \bigr)\\
&
\quad \quad \quad \quad - \displaystyle f(V^{(y,k)}(t),\psi^{(k)}(t))\cdot M_{t}\bigl(V^{(y,k)}(\cdot),\psi^{(k)}
(\cdot) \bigr)\Bigr|\Bigr]\\
& \quad \le \displaystyle \|f\| \cdot \E \left[ \left|M_{t}\bigl(V^{(x,k)}(\cdot),\psi^{(k)} (\cdot) \bigr)-
M_{t}\bigl(V^{(y,k)}(\cdot),\psi^{(k)}
(\cdot) \bigr)\right| \right]\\
& \quad \quad + \displaystyle 2 \|f\|\cdot \E\Big[M_{t}\bigl(V^{(y,k)}(\cdot),\psi^{(k)} (\cdot) \bigr) I_{\{ |f(V^{(x,k)}(t),\psi^{(k)} (t))-f(V^{(y,k)}(t),\psi^{(k)} (t))| \ge \e\}}\Big]\\
& \quad \quad + \displaystyle \varepsilon \cdot \E
\big[M_{t}\bigl(V^{(y,k)}(\cdot),\psi^{(k)} (\cdot) \bigr)\big]\\
& \quad = \displaystyle \hbox{(I)}+ \hbox{(II)}+ \hbox{(III)},
\end{aligned}\end{equation} where $\|f\|:=\sup \{|f(x,k)|: (x,k) \in \R^{d} \times \ss\}$.
From Lemma \ref{4.3} term (I) in (\ref{(4.26)})
tends to zero as $|x-y| \to 0$. From Lemmas \ref{lem:GFP6} and \ref{4.4}, we derive that term (II) in (\ref{(4.26)})
also tends to zero as $|x-y| \to 0$. Meanwhile, term (III) in (\ref{(4.26)}) can be arbitrarily small since the
multiplier $\varepsilon$ is arbitrary and $M_{t}\bigl(V^{(y,k)}(\cdot),\psi^{(k)} (\cdot) \bigr)$ is integrable by
Lemma~\ref{4.4}. The proof is completed. \end{proof}

\section{Exponential Ergodicity}\label{sect:exp-ergodicity}

In this section, we follow \cite{Xi-09} and investigate the exponential ergodicity for the process $(X,\La)$. To this end, let us first  recall some relevant terminologies. As in \cite{MeynT-93II}, the process $(X,\La)$ is called {\em bounded in probability on average} if for each $(x,k)\in \R^{d}\times \ss$ and each $\e > 0$, there exists a compact subset $C \subset \R^{d}$ and a finite subset $N \subset \ss$ such that
\begin{displaymath}
\liminf_{t\to\infty} \frac{1}{t}\int_{0}^{t} P(s,(x,k), C\times N)\d s \ge 1-\e.
\end{displaymath}
We now introduce a Foster-Lyapunov drift condition as follows. For some $\alpha, \beta > 0$,  $f(x,k) \ge 1$, a compact subset $C \subset \R^{d}$ and a finite subset $N \subset \ss$, and a nonnegative function $V(\cdot, \cdot) \in C^{2}(\R^{d}\times \ss)$,
\begin{equation}
\label{eq-FL-drift}
\A V(x,k) \le -\alpha f(x,k) + \beta \one_{C\times N}(x,k), \qquad (x,k) \in \R^{d}\times \ss,
\end{equation}
where $\A$ is the operator defined in \eqref{eq-operator}.

\begin{Proposition}\label{prop-bdd-in-prob}
Suppose \eqref{eq-FL-drift} and Assumptions \ref{I1},  \ref{qH}, \ref{qkappa}, \ref{StFP1}, and  \ref{assumption-finite-range}  hold. Then the process $(X,\La)$ is bounded in probability on average and possesses an invariant probability $\pi$.
\end{Proposition}
\begin{proof} By Theorem \ref{thm:str-Feller}, the process $(X,\Lambda)$ is strong Feller and hence a $T$-process in the terminology of \cite{MeynT-93III}. In addition, Proposition \ref{prop-EU1}  indicates that $(X,\Lambda)$ is non-explosive. Therefore Theorem 4.7 of  \cite{MeynT-93III} implies that $(X,\La)$ is bounded in probability on average. The assertion that $(X,\La)$ possesses an invariant probability $\pi$ is a direct consequence of  \cite[Theorem 4.5]{MeynT-93III}.
\end{proof}


For any positive function $f: \R^{d}\times \ss\mapsto [1,\infty)$ and any signed measure $\nu$ defined on $\B(\R^{d}\times \ss)$, we write
\begin{displaymath}
\|\nu\|_{f} : = \sup\{|\nu(g)|: g \in \B(\R^{d}\times \ss) \text{ satisfying } |g| \le f \},
\end{displaymath} where $\nu(g) : = \sum_{l\in\ss}\int_{\R^{d}} g(x,l)\nu(\d x,l)$ is the integral of the function $g$ with respect to the measure $\nu$.
Note that the usual total variation norm $\|\nu\|$ is just $\|\nu\|_{f}$ in the special case when $f\equiv 1$. For a function $\infty > f \ge 1 $ on $\R^{d}\times \ss$, the process $(X,\La)$ is said to {\em $f$-exponentially ergodic}  if there exists a probability measure $\pi(\cdot) $, a constant $\theta  \in (0,1)$ and a finite-valued function $\Theta(x,k)$ such that
\begin{equation}
\label{eq-exp-ergodicity-defn}
\norm{P(t,(x,k),\cdot) - \pi(\cdot)}_{f} \le \Theta(x,k) \theta^{t}
\end{equation}
for all $t\ge 0 $ and all $(x,k) \in \R^{d}\times \ss$.

We need the following assumption:
\begin{Assumption}\label{Assumption-irreducibility}
 For any distinct $k, l \in \ss$, there
exist $k_{0}$, $k_{1}$, $\cdots$, $k_{r}$ in $\ss$ with $k_{i} \ne
k_{i+1}$, $k_{0}=k$ and $k_{r}=l$ such that the set $\{x\in \R^{d}:
q_{k_{i}k_{i+1}}(x)>0\}$ has positive Lebesgue measure for $i=0, 1, \dots, r-1$.
\end{Assumption}
\begin{Theorem}\label{thm-exp-ergodicity}
Suppose Assumptions \ref{I1},  \ref{qH}, \ref{qkappa},  \ref{StFP1},    \ref{assumption-finite-range}, and \ref{Assumption-irreducibility}  hold. In addition, assume that there exist positive numbers $\alpha, \gamma$ and a nonnegative function $V\in C^{2}( \R^{d}\times\ss) $ satisfying
\begin{itemize}
  \item[(i)] $V(x,k)\to \infty$ as $|x|\vee k \to \infty$, 
  \item[(ii)] $\A V(x,k) \le -\alpha V(x,k) + \gamma$, $(x,k) \in \R^{d}\times \ss$.

\end{itemize} Then the process $(X,\La)$ is $f$-exponentially ergodic with $f(x,k)= V(x,k) +1$.
\end{Theorem}

\begin{proof} Note that the existence of $V$ satisfying (i) and (ii) in the statement of the theorem trivially leads to   \eqref{eq-FL-drift},  and hence, together with the other assumptions of the theorem,  the conclusions of Proposition \ref{prop-bdd-in-prob}.
We next show that the process $(X,\La)$ is irreducible in the sense that for any $t>0$, $(x,k) \in \R^{d}\times \ss$, $A \in \B(\R^{d})$ with positive Lebesgue measure, and  $ l \in \ss$, we have
$P(t,(x,k), A\times \{l\}) > 0 $. To this end,   for each $k\in \ss$, we kill the L\'evy process $X^{(k)}$ of  \eqref{(EU1)} with killing rate $q_{k}(\cdot)$. Denote by $P^{(k)} (t,x,\cdot)$ the transition probability of the killed process. Then we have
\begin{align}  \nonumber P & (t,(x,k),A\times \{l\}) \\ \nonumber
 & =\delta_{kl} P^{(k)}(t,x,A)+\sum_{m=1}^{+\infty} \ \ \idotsint\limits_{0<t_{1}<t_{2}<\cdots
<t_{m}<t}
  \sum_{{l_{0}, l_{1}, l_{2}, \cdots, l_{m} \in
\ss}\atop{l_{i}\neq l_{i+1}, l_{0}=k, l_{m}=l}}\int_{\R^{d}} \cdots
\int_{\R^{d}}P^{(l_{0})}(t_{1},x,\d y_{1}) q_{l_{0}l_{1}}(y_{1})\\ \label{(FP22)}
&\ \quad \times  P^{(l_{1})}(t_{2}-t_{1},y_{1},\d y_{2})\cdots
q_{l_{m-1}l_{m}}(y_{m})P^{(l_{m})}(t-t_{m},y_{m},A) \d t_{1} \d t_{2}
\cdots \d t_{m},
\end{align}  where $\delta_{kl}$ is the Kronecker
symbol in $k$, $l$, which equals $1$ if $k=l$ and  $0$ if $k\neq
l$. As argued in \cite{Xi-09}, Assumption \ref{StFP1} guarantees that each term of the form $P^{(l)} (s,x, A)$ with $l\in \ss, s > 0$ and $ x\in \R^{d}$ is positive; this, together with Assumption  \ref{Assumption-irreducibility}, implies that $P(t,(x,k), A\times\{l\}) > 0$.

Using the same argument as that in the proof of Theorem 6.3 of \cite{Xi-09}, we can   show that all compact subsets of $\R^{d}\times \ss$ are petite for the skeleton $\{X(nh),\La(nh)), n\ge 0\}$. Then the desired $f$-exponential ergodicity follows from   Theorem 6.1 in \cite{MeynT-93III}.\end{proof}

\begin{Example}
In this example, we consider a coupled one-dimensional Ornstein-Uhlenbeck process
\begin{equation}
\label{eq-OU}
\d X(t) = \alpha(\La(t) )X(t)\d t + \sigma(\La(t) ) \d B(t) + \int_{\R\setminus\{0\}} \beta(\La(t-)) z N(\d t, \d z),
\end{equation}
where   for each $k \in \ss = \{0, 1, 2, \dots\}$, $\alpha_{k}:=\alpha(k),  \beta_{k}:=\beta(k)$, and  $\sigma_{k}:=\sigma(k)$ are real numbers to be determined later, $B$ is a standard one-dimensional Brownian motion,  and $N(\d t, \d z)$ is Poisson random measure with characteristic measure $\nu(\d z) = \frac{1}{2}e^{-|z|} \d z$. Suppose the switching component $\La$ is generated by the $q$-matrix
\begin{equation}
\label{eq-Q(x)-OU-example}
Q(x) = \begin{pmatrix}
    -q_{01}(x) & q_{01}(x)         & 0  & 0 & 0  & 0 &  \dots \\
        q_{10}(x)  &  -(q_{10}(x) + q_{12}(x) )&       q_{12}(x)     &       0       & 0  &    0  &   \dots\\
           0            &  q_{21}(x) & - (q_{21}(x) + q_{23}(x) ) &  q_{23}(x)  & 0  &   0 &    \dots \\
     \vdots & \vdots& \vdots & \vdots & \vdots & \vdots &  \ddots
     \end{pmatrix},
\end{equation}  where $q_{k,k-1}(x) $ and $q_{k,k+1}(x)$ are positive and Lipschitz    continuous functions.  Obviously, Assumptions \ref{I1},  \ref{qH}, \ref{qkappa},  \ref{StFP1}, \ref{assumption-finite-range}, and \ref{Assumption-irreducibility}  hold.

Let us consider the functions $V(x,k) = (k+1)x^{2}$ for $(x,k) \in \R\times \ss$. Then we have
\begin{align*}
\A V(x,0)  & = [2 \alpha_{0}- q_{01}(x) + 2q_{01}(x)] x^{2}  + \sigma_{0}^{2}+ \int_{\R\setminus\{0\}} [(x+ \beta_{0} z)^{2} - x^{2}] \nu(\d z) \\
 &  =  [2 \alpha_{0}+  q_{01}(x) ] x^{2} + \sigma_{0}^{2} +  4 \beta_{0}^{2},    
\end{align*}
and for $k =1, 2, 3,\dots$,
\begin{align*}
 \A V(x,k) & = [2 (k+1)\alpha_{k} +k q_{k,k-1} (x)  -(k+1) (q_{k,k-1} (x)+ q_{k,k+1}(x) )+ (k+2) q_{k,k+1}(x) ]x^{2}\\
 & \qquad + (k+1) \sigma_{k}^{2}  + \int_{\R_{0}} [(k+1)(x+ \beta_{k} z)^{2} - (k+1) x^{2}] \nu(\d z)\\
 &  = [2 (k+1)\alpha_{k}  -q_{k,k-1} (x) + q_{k,k+1}(x) ]x^{2} + (k+1) \sigma_{k}^{2}  +  4 (k+1) \beta_{k}^{2}. 
\end{align*}

Now 
suppose there exist positive constants $K_{1}$ and $K_{2}$ so that the following conditions are satisfied:
\begin{itemize}
  \item[(a)] $2 \alpha_{0}+  q_{01}(x)  \le - K_{1} < 0 $,
  \item[(b)] for each $k \in \ss$, we have $\sigma_{k} > 0$,  and $  (k+1) \sigma_{k}^{2}  +  4 (k+1) \beta_{k}^{2} \le K_{2}< \infty$,
  \item[(c)] for all $k\in \ss\setminus\{0\}$, we have  $2 (k+1) \alpha_{k}  -q_{k,k-1} (x) + q_{k,k+1}(x)  \le -K_{1} (k+1) < 0$. 
\end{itemize} 
Then it follows that for all $(x,k) \in \R\times \ss$, we have $$\A V(x,k)  \le -K_{1}  (k+1) x^{2} + K_{2}= -  K_{1} V(x,k) +  K_{2}. $$ This verifies conditions (i) and (ii) of Theorem \ref{thm-exp-ergodicity}. Hence we conclude that the process $X$ of \eqref{eq-OU}
is $f$-exponentially ergodic.

Note that we can choose $\alpha_{k}, \beta_{k}, \sigma_{k}$ and $Q(x)$ so that: (i) $X^{(0)}$ is exponentially ergodic, (ii) $X^{(k)}$ is transient for $k =1, 2, \dots$, but (iii) the process $(X,\La)$ of \eqref{eq-OU}
is $f$-exponentially ergodic.   \end{Example}

To proceed, we assume in the rest of the section that
\begin{Assumption}\label{assumption-infty}
For each $i \in \ss$, there exist $b(i)$, $\sigma_j(i)\
\in \R^{d\times d}$,  $ j=1,2,\dots,d$,
such
that as $|x|\to \infty$, \begin{equation}\label{con-infty}\begin{aligned}
&  \frac{b(x,i)}{|x|}=b(i) {x\over |x|}+ o(1),\\
& \frac{\sigma(x,i)}{|x|}=(\sigma_1(i)x,\sigma_2(i)x,\dots,\sigma_d(i)x)
{1\over |x|}+o(1),\\
\end{aligned}\end{equation} where $o(1)\to 0$ as $|x|\to
\infty$.
\end{Assumption}

\begin{Proposition}\label{linear} Suppose  Assumptions  \ref{I1},  \ref{qH}, \ref{qkappa},  \ref{StFP1},    \ref{assumption-finite-range},  \ref{Assumption-irreducibility}, and \ref{assumption-infty} hold. Assume that for each $i\in \ss$ and some  $p \in (0, 2)$  
such that as $|x|\to \infty$, we have
\begin{equation}\label{eq-jumpcondition-infty}
\int_{U} \biggl( \frac{|x + c(x,i,u)|^{p}}{|x|^{p}} -1\biggr) \Pi(\d u) \le  \wdh c_{i} < \infty.
\end{equation} Denote  $\mu_{i}: = \lambda_{\max} ( \frac{b(i) + b(i)'}{2}  + \sum_{j=1}^{d} \sigma_{j}(i) \sigma_{j}(i)' ) + \wdh c_{i} $ for each $i\in \ss$. Suppose there exist $\al >0$ and $g_{i}> 0$, 
$i \in \ss$  such that $g_{i} \to \infty$ as $i \to \infty$ and when $x$ is sufficiently large,
\begin{equation}
\label{eq-Q-drift-condition}
\sum_{j\in \ss} q_{ij}(x) g_{j} + p(\al + \mu_{i}) g_{i} \le 0    \, \,
\text{ for all } \, \, i \in \ss, 
\end{equation} where  $p\in (0,2)$ is as in \eqref{eq-jumpcondition-infty}. 
Then $(X,\La)$ is $f$-exponential ergodic.
\end{Proposition}

\begin{proof} Let $p \in (0, 2)$, $\al > 0$ and $g_{i}, i\in \ss$ be as in the statement of the proposition. Let the function $V(x,i) \in C^{2}(\R^{d}\times \ss)$
and $V(x,i)= g_{i} |x|^{p}$
when $(x,i) \in \bigl(\R^{d}\setminus\{y: |y|\le 1\}\bigr)\times \ss$.
  It is readily seen that  for each $i\in\ss$,
$V(\cdot,i)$ is continuous, nonnegative, and converges to $\infty$ as $|x| \vee i \to \infty$. Detailed calculations reveal that for $x\not=
0$, we have
\begin{align*}
& D (|x|^{p}) =  p |x|^{p-2}x,\\
& D^2  (|x|^{p}) = p \big[|x|^{p -2}I + (p-2)|x|^{p -4 }xx'\big].
\end{align*} Meanwhile, it follows from \eqref{con-infty}  that  when  $|x|\to \infty$
\begin{displaymath}
a(x,i)=\sigma(x,i)\sigma'(x,i) = \sum_{j=1}^d \sigma_j(i)xx'\sigma'_j(i) +
o(|x|^2).
\end{displaymath}
  Then for all $(x,i ) \in \R^{d}\times \ss$ with $|x|$ sufficiently large,  detailed computations using  Assumption \ref{assumption-infty} reveal that     \begin{align*}
  \A V(x,i)   &    = pg_{i}|x|^{p} \Biggl[ \frac{x' b(i) x}{|x|^{2}} + \frac{ \sum_{j=1}^{d} x' \sigma_{j}(i) \sigma_{j}(i)' x}{|x|^{2}} + (p-2)  \frac{(x'\sigma_{j}(i)' x)^{2}}{|x|^{4}}  \\
  & \qquad\qquad\  + \int_{U} \biggl( \frac{|x+ c(x,i,u)|^{p}}{|x|^{p}} -1 \biggr)\Pi(\d u) + \sum_{j\in \ss}  q_{ij}(x) \frac{g_{j}}{pg_{i}}   + o(1)  \Biggr].
\end{align*}
Notice that \begin{displaymath}
\frac{x' b(i) x}{|x|^{2}} + \frac{ \sum_{j=1}^{d} x' \sigma_{j}(i) \sigma_{j}(i)' x}{|x|^{2}}  \le \lambda_{\max} \biggl( \frac{b(i) + b(i)'}{2}  + \sum_{j=1}^{d} \sigma_{j}(i) \sigma_{j}(i)'\biggr).
\end{displaymath}
Also since $0 < p < 2$, we have $(p-2)  \frac{(x'\sigma_{j}(i)' x)^{2}}{|x|^{4}} \le 0$.   
Therefore  for  $|x|$ sufficiently large, it follows from \eqref{eq-jumpcondition-infty}  and \eqref{eq-Q-drift-condition} 
that  \begin{align*}
\A V(x,i)  & \le pV(x,i) \biggl[ \lambda_{\max} \biggl( \frac{b(i) + b(i)'}{2}  + \sum_{j=1}^{d} \sigma_{j}(i) \sigma_{j}(i)'\biggr) + \wdh c_{i} + \sum_{j\in \ss}   q_{ij} (x)\frac{g_{j}}{pg_{i}}      + o(1)  \biggr] \\
   & =  pV(x,i) \biggl[ \mu_{i}+\sum_{j\in \ss}   q_{ij}(x) \frac{g_{j}}{pg_{i}}  +  o(1)  \biggl] \le pV(x,i) [- \al + o(1)].
\end{align*}
In particular, we can choose $R>0$ sufficiently large   so that
\begin{displaymath}
\A V(x,i) \le - \frac{\al}{2}  pV(x,i), \text{ for all } (x,i) \in \R^{d}\times \ss \text{ with }|x|\ge R.
\end{displaymath} Next we choose $\gamma > 0$ sufficiently large so that
\begin{displaymath}
\A V(x,i) \le  - \frac{\al}{2}  pV(x,i) + \gamma, \text{ for all } (x,i) \in \R^{d}\times \ss.
\end{displaymath}
This verifies condition (ii) of Theorem \ref{thm-exp-ergodicity}. Therefore the desired assertion on $f$-exponential ergodicity  follows.
\end{proof}

\section{Concluding Remarks}\label{sect:ConRemark}

Motivated by the increasing need of modeling complex systems, this paper is devoted to the investigation of a class of  regime-switching jump diffusions with countable regimes. By using an interlacing procedure together with an exponential killing technique, this paper was able to establish   the existence and uniqueness of a strong solution to the associated stochastic differential equations under more general formulation than those in the literature.  The paper next used coupling method and an appropriate  Radon-Nikodym derivative to  derive Feller and strong Feller properties and exponential ergodicity for such processes.

A number of other problems deserve further investigation. In particular, in view of Yamada and Watanabe's work on the uniqueness of solutions of stochastic differential equations (\cite{YamaW-71}),  one may naturally ask whether the Lipschitz condition
 can be relaxed. 
 Also of interest is to consider the problem of successful couplings 
 for regime-switching jump diffusions.

\para{Acknowledgement}.\quad
The authors would like to express their appreciation to the referees for their careful reading of the manuscript and  helpful suggestions for improvements. 




\bibliographystyle{apalike}


\end{document}